\documentclass[a4paper,leqno,twoside]{amsart}%
\usepackage{amsmath,amssymb,amsthm,enumerate,hyperref,color,tikzsymbols}

\usepackage[normalem]{ulem}

\numberwithin{equation}{section}
\newtheorem{theorem}{Theorem}[section]

\newtheorem{proposition}[theorem]{Proposition}
\newtheorem{lemma}[theorem]{Lemma}
\newtheorem{corollary}[theorem]{Corollary}
\newtheorem{remark}[theorem]{Remark}

\def\sideremark#1{\ifvmode\leavevmode\fi\vadjust{\vbox to0pt{\vss
			\hbox to 0pt{\hskip\hsize\hskip1em
				\vbox{\hsize2.1cm\tiny\raggedright\pretolerance10000
					\noindent #1\hfill}\hss}\vbox to15pt{\vfil}\vss}}}%

\newif\ifcomment \commentfalse
\def\commentON{\commenttrue}

\long\outer\def\BC#1\EC{\ifcomment \sloppy \par \# \ldots\dotfill
	{\em #1} \dotfill \# \par \fi } \commentON

\newcommand{\remove}[1]{}

\definecolor{cadmiumgreen}{rgb}{0.0, 0.42, 0.24}
\definecolor{darkgreen}{rgb}{0.0, 0.5, 0.2}
\definecolor{purple}{rgb}{0.5, 0.0, 0.5}

\def\a{\alpha}

\def\d{\delta}
\def\e{{\varepsilon}}
\def\tuip{u^i_{p}}
\newcommand{\R}{\mathbb{R}}
\newcommand{\loc}{\mathop{\mathrm{loc}}}
\def\L{\Lambda}

\title[Morse index computation for {H\'e}non problem ]{Morse index computation for radial solutions of the {H\'e}non problem in the disk}
\author{Anna Lisa Amadori}
\address[Anna Lisa Amadori]{Dipartimento di Scienze Applicate, Universit\`a di Napoli \textquotedblleft Parthenope\textquotedblright, Centro Direzionale di Napoli, Isola C4, 80143 Napoli, Italy.}
\email{annalisa.amadori@uniparthenope.it}
\author{Francesca De Marchis}
\address[Francesca De Marchis]{Dipartimento di Matematica {\sl Guido Castelnuovo}, Sapienza Università di Roma, Piazzale Aldo Moro 5, 00185 Roma}
\email{demarchis@mat.uniroma1.it}
\author{Isabella Ianni}
\address[Isabella Ianni]{Dipartimento SBAI, Sapienza Università di Roma, via Scarpa 10, 00161 Roma}
\email{isabella.ianni@uniroma1.it}



\thanks{The last author is partially supported by:  PRIN $2017$JPCAPN$\_003$ grant,   VALERE:\emph{Vain-Hopes} grant, INDAM - GNAMPA}

\begin{document}

\begin{abstract}
We compute the Morse index $\textsf{m}(u_{p})$ of any radial solution $u_{p}$ of the semilinear  problem:
\begin{equation}
\label{problemaAbstract}\tag{P}
\left\{
\begin{array}{lr}
-\Delta u=|x|^{\alpha}|u|^{p-1}u & \mbox{in } B\\
u=0 & \mbox{ on }\partial B
\end{array}
\right.
\end{equation}
where $B$ is the unit ball of $\mathbb R^{2}$ centered at the origin, $\alpha\geq 0$ is fixed and $p>1$  is sufficiently large. In the case $\alpha=0$, i.e. for the \emph{Lane-Emden problem}, this leads to the following Morse index formula
\[\textsf{m}(u_{p}) = 4m^{2}-m-2, \] 
for $p$ large enough, where  $m$ is the number of nodal domains of $u$.
\end{abstract}

\maketitle

\section{Motivations and main results} \label{section:intro}
We consider the following classical semilinear elliptic problem
\begin{equation}
\label{problemaHenonIntro}
\left\{
\begin{array}{lr}
-\Delta u=|x|^{\alpha}|u|^{p-1}u & \mbox{in } B\\
u=0 & \mbox{ on }\partial B
\end{array}
\right.
\end{equation}
where $\alpha\geq 0$, $p>1$ and $B$ is the unit ball of $\mathbb R^{N}$, $N \geq 2$, centered at the origin.

\

When $\alpha>0$ \eqref{problemaHenonIntro}  has been introduced by H\'enon in \cite{Hen73} in the study of stellar clusters thus it is known as the \emph{H\'enon problem},  when $\alpha=0$ \eqref{problemaHenonIntro} reduces to the classical  \emph{Lane-Emden problem}.

\

From a mathematical point of view it is well known that, for any fixed $\alpha\geq 0$, problem \eqref{problemaHenonIntro} admits solutions,
and in particular radial solutions, for every $p>1$  if $N=2$, and for every $p \in (1, p_{\alpha})$ if $N\geq 3$, where 
$p_{\alpha} = \frac{N+2+2\alpha}{N-2} $  (see \cite{Ni82}). Moreover for
any given $m \geq 1 $ there is exactly one couple of radial solutions of \eqref{problemaHenonIntro} which have exactly $m$ nodal zones, they are classical solutions and they are one the opposite of
the other (see for instance \cite{BW93, NiN1985, Kajikiya}).
\

Observe  that the two  problems ($\alpha=0$ and $\alpha>0$) have a strong correlation, indeed the change of variable 
\begin{equation}
\label{changeIntro}
v(t)= \left(\frac{2}{2+\alpha}\right)^{\frac{2}{p-1}}  u(r),\quad t=r^{\frac{2+\alpha}{2}},
\end{equation}
transforms radial solutions $u$ of the H\'enon problem in dimension $N$ into radial solutions $v$  of the Lane-Emden problem in dimension  $M = M(N,\alpha) := \frac{2(N+ \alpha)}{2+\alpha}$, with the same number of zeros. Notice that $M=N$ when $N=2$, while  $M<N$ for any $N\geq 3$  and in this case $M$  may be a non integer extended dimension.

\

This paper deals with the computation of the Morse index of all the radial solutions of \eqref{problemaHenonIntro} in dimension $N=2$, for any $\alpha\geq 0$ fixed  and for large values of the exponent $p$.

\

We recall that the Morse index $\textsf{m}(u)$ of a solution $u$ of \eqref{problemaHenonIntro} is the maximal dimension of a subspace $X\subset  H^{1}_{0}(B)$ where the quadratic form $Q_{u}:H^{1}_{0}(B)\times H^{1}_{0}(B)\rightarrow \mathbb R$
\[Q_{u}(v,w)=\int_{B}\big( \nabla v\nabla w-|x|^{\alpha}p|u|^{p-1}vw\big)\,dx\]
is negative definite. Equivalently, since $B$ is a bounded domain, $\textsf{m}(u)$ can be defined as the number of the negative Dirichlet eigenvalues of  the linearized operator at $u$
\[L_{u} = -\Delta - |x|^{\alpha} p|u|^{p-1}\] counted with their multiplicity.

\

The knowledge of the Morse index    has important applications: it allows to distinguish and classify solutions and to study their stability properties. Moreover it is well known that a change in the Morse index may imply bifurcation, which may also give rise to symmetry breaking phenomena (\cite{GI20, Ama20bif, AG14, KW19, FN19}).

Focusing on radial solutions $u_{p}$ of problem \eqref{problemaHenonIntro}, it is known, from \cite{HRS11, BW03} in the case $\alpha=0$  and \cite{AG-sing2} in the case $\alpha>0$, that the \emph{radial Morse index} $\textsf{m}_{\mathrm{rad}}(u_{p})$ (i.e. the number of the negative eigenvalues of $L_{u_{p}}$ in the subspace $H^{1}_{0,\mathrm{rad}}(B)$ of the radial functions in $H^{1}_{0}(B)$), coincides with the number $m$ of nodal zones of $u_{p}$:
\[\textsf{m}_{\mathrm{rad}}(u_{p})=m\]
and moreover the solution $u_{p}$ is radially nondegenerate. 
Nevertheless the complete Morse index of a radial solution $u_{p}$ is generally higher, and indeed the following lower bound holds true
\begin{equation}
\label{MorseLoweBound}
\textsf{m}(u_{p})\geq \left\{\begin{array}{lr}
m+(m-1)N & \mbox{ if } \alpha\in [0,2)\\
m+ (m-1)\bigg(N +  \displaystyle\sum_{j=1}^{[\frac{\alpha}{2}]}N_{j+1}\bigg)& \mbox{ if } \alpha\geq 2
\end{array}\right.
\end{equation}
as proved
in \cite{DIP-N>3} for the case $\alpha=0$ (see also \cite{AP04, BdG06} for previous results in this direction) and then for the case $\alpha>0$ in \cite{AG-sing2}  by exploiting the relation in \eqref{changeIntro} (see also \cite{dSP17}). Here $N$ stands for the dimension, $N_{j}=\frac{(N+2j-2)(N+j-3)!}{(N-2)!j!}$ is the multiplicity of the $j$-th eigenvalue $\lambda_{j}=j(N+j-2)$ of the Laplace-Beltrami operator on the sphere $\mathbb S_{N-1}$ and $[\cdot]$ is the integer part.\\
Observe that, by Morse index comparison, one deduces from the estimates \eqref{MorseLoweBound} that a least energy nodal (i.e. $m\geq 2$) solution for problem \eqref{problemaHenonIntro}, having  Morse index $2$ (cfr. \cite{BW03}), can not be radial (see \cite{AG-sing2} and also \cite{AP04, BdG06, dSP17}).

\

For the Lane-Emden problem ($\alpha=0$) and in dimension $N\geq 3$  the estimate \eqref{MorseLoweBound} is surprisingly optimal, indeed in \cite{DIP-N>3}  the following Morse index formula has been proven:
\begin{equation}
\label{LaneEmdenFormula}\textsf{m}(u_{p})=m+(m-1)N,\qquad \mbox{ for }p\in [\bar p, p_{\alpha}),
\end{equation}
for a certain $\bar p:=\bar p(m,N)>1$.\\
This formula has been then generalized to the H\'enon case ($\alpha>0$) in \cite{AG_N3}, obtaining, again in dimension $N\geq 3$, that
	\begin{equation}\label{HenonFormulapcrit} \textsf{m}(u_{p})  = m + (m-1) \sum\limits_{j=1}^{\left[\frac{2+ \alpha}{2}\right]} N_j + \sum\limits_{j=1}^{\left\lceil\frac{\alpha}{2}\right\rceil} N_j
	\end{equation}  
 for $p\in [\bar p, p_{\alpha})$, where $\bar p:=\bar p(m,N,\alpha)>1$. Here $[\cdot]$ is the integer part and $\lceil \cdot \rceil$ the ceiling function. Observe that for any $\alpha>0$ the Morse index in this formula is actually higher then the lower value found in \eqref{MorseLoweBound};  in particular \eqref{HenonFormulapcrit} implies, again by Morse index comparison, that the ground state (positive) solution of the H\'enon problem, which has Morse index $1$, is not radial for $p\in [\bar p, p_{\alpha})$. Indeed
for the positive (i.e. $m=1$)  radial solution $u_{p}$  \eqref{HenonFormulapcrit} gives 
 \[\textsf{m}(u_{p})\geq 1+N\ (>1)\]
 (see also \cite{SSW02}, where the same conclusion is derived via energy comparison). 

\

Formulas \eqref{LaneEmdenFormula} and \eqref{HenonFormulapcrit} have been derived both  by the study of an \emph{auxiliary singular eigenvalue problem} associated to the linearized operator $L_{u_{p}}$ which, in the radial setting,  can be decomposed into a radial and an angular part. In particular, the study of the radial part strongly depends on the qualitative properties of the solution $u_{p}$, and the proofs of both the formulas specifically exploit the knowledge of the asymptotic behavior of $u_{p}$ as $p\rightarrow p_{\alpha}$ from the left.
\\ In dimension $N\geq 3$ this behavior is indeed well known:  all the radial solutions of \eqref{problemaHenonIntro} blow-up at the origin as $p\rightarrow p_{\alpha}$ and vanish elsewhere, moreover each radial solution with $m$ nodal zones is a \emph{tower of $m$ bubbles}, i.e., in short, it looks like $m$ superpositions of the same limit profile 
\begin{equation}\label{bubbleN}U_{\alpha}(x)= \left(1+\frac{|x|^{2+\a}}{(N+\alpha)(N-2)}\right)^{-\frac{N-2}{2+\a}} ,\end{equation} with alternate sign and scaled with different speeds (for $\alpha=0$ see for instance \cite{AP87, DIP-N>3, GST20}, for $\alpha >0$ see \cite{AG14,AG_N3}).
Observe that  $U_{\alpha}$ is a solution of the critical equation 
\begin{equation}\label{eqCriticaN3}-\Delta U_{\alpha} =|x|^{\alpha}U_{\alpha}^{p_{\alpha}},\qquad  x\in \mathbb R^{N}.\end{equation} 

\

In this paper we focus on the $2$-dimensional case and derive the analogous of formulas \eqref{LaneEmdenFormula} and \eqref{HenonFormulapcrit}. 

\

In dimension $N=2$  the asymptotic behavior of the radial solutions of  \eqref{problemaHenonIntro} as $p\rightarrow +\infty$ (in this case the exponent $p_{\alpha}$ is substituted with $+\infty$) is different: one can show that all these solutions do not blow-up but concentrate at the origin and vanish elsewhere. Moreover, since $p_{\a}=+\infty$, the bubbling behavior is more delicate to be described and indeed  profiles different than the solutions of \eqref{eqCriticaN3} are involved, as shown in \cite{GGP14, AG-N2} for the solution with $m=2$ nodal regions. \\Very recently in \cite{IS-arxiv} the results in \cite{GGP14, AG-N2} have been extended  to all the radial solutions of \eqref{problemaHenonIntro}, showing that the radial solution $u_{p}$ with $m$ nodal zones (for any $m\geq 1$) develops a  \emph{tower of $m$ bubbles}, one in each nodal zone, similarly as in dimension $N\geq 3$ but, unlike the higher dimensional case, the profile of each bubble is now different, and given by \begin{equation}\label{bubblecompletaINTRO} Z_{\alpha,i}(x) = \log \frac{2\theta_{i}^2\gamma_{i} \, |x|^{\frac{(\alpha+2)}{2}(\theta_{i}-2)}}{(\gamma_{i} + |x|^{\frac{(\alpha+2)}{2}\theta_{i}})^2},\quad \mbox{ with }\ \ 
\gamma_{i}= \frac{\theta_{i}+2}{\theta_{i}-2} \left(\frac{\theta_{i}^2-4}{2}\right)^{\frac{\theta_{i}}{2}},\end{equation}
 for $i=0,\ldots, m-1$, where the sequence $(\theta_{i})_{i\in\mathbb N}$ is uniquely determined by the following iteration 
 \begin{equation}\label{introtheta}
 \left\{\begin{array}{lr}
 \theta_0=2\\
 \theta_i=\frac{2}{\mathcal L\left[\frac{2}{2+\theta_{i-1}}e^{-\frac{2}{2+\theta_{i-1}}}\right]}+2 & \mbox{ for } i\geq 1
 \end{array}
 \right.
 \end{equation}
($\mathcal  L$ is the Lambert function) and 
$Z_{\alpha,i}$ is a radial solution of the singular Liouville equation
\begin{equation}\label{singLiouvilleINTRO}-\Delta Z_{\alpha,i}=\left(\frac{\alpha +2 }{2}\right)^2|x|^{\alpha}e^{Z_{\alpha,i}} + 	 (\alpha+2)\pi (2-\theta_{i})\delta_{0}\qquad \mbox{ in } \mathbb R^{2},\end{equation}
see Section \ref{Section:Notazioni&Idee} for more details.


\

As a consequence of this sharp asymptotic analysis one expects that in dimension $N=2$ formulas \eqref{LaneEmdenFormula} and \eqref{HenonFormulapcrit} do not hold, and that the constants $\theta_{i}$'s  must be involved in the Morse index computations, for large values of $p$.

\
 
Indeed this is exactly what has been observed in the case of the radial solution $u_{p}$ with $m=2$ nodal zones, whose Morse index has been computed  in 
 \cite{DIP-N2} for the Lane-Emden problem ($\alpha=0$) and  in \cite{AG-N2}  for the H\'enon problem ($\alpha>0$). The results in \cite{DIP-N2, AG-N2} may be summarized as follows
 \begin{equation}\label{2zone}
  	\textsf{m}	(u_p)   
\begin{cases}
\displaystyle = 2+  2\left\lceil\frac{\alpha}{2}\right\rceil + 2\left[\frac {2+\a}4 \theta_{1}\right] 
 & \text{ if } \frac {2+\a}4 \theta_{1}  \notin \mathbb N   \\
 \displaystyle 	\in \left[ 2\left\lceil\frac{\alpha}{2}\right\rceil + \frac {2+\a}2 \theta_{1}  ,\ 2 + 2\left\lceil\frac{\alpha}{2}\right\rceil + \frac {2+\a}2 \theta_{1} \right] 
  	& \text{ otherwise} 
 	\end{cases}
\end{equation}
 for $p\geq\bar p(\a)\,(>1)$. 

\

So far  in dimension $N=2$ the value of the Morse index for all the  radial solutions $u_p$ of \eqref{problemaHenonIntro} with any number $m> 2$ of nodal zones, for $p$ large, was unknown.  
Here we fill in this gap showing that

\begin{theorem}\label{theorem_Main_Intro}
Let $N=2$, $\alpha\geq 0$ and let $u_{p}$ be a radial solution to \eqref{problemaHenonIntro} with $m$ nodal zones. Let $(\theta_i)_{i\in\mathbb N}$ be the sequence in \eqref{introtheta}. Then there exists $\bar p=\bar p(m,\alpha)>1$ such that 	 for $p\geq\bar p$
 	 \begin{equation}\label{formulaHenonIntro1} 
 	\textsf{m}(u_{p}) = m + 2 \left\lceil \frac{\a}2\right\rceil + 2 \sum\limits_{i=1}^{m-1}\left[ \frac{2\!+\!\a }4 \theta_i \right] 
 	\end{equation} if $\frac {2+\a}4 \theta_i \notin \mathbb N$, for every $i=1,\dots, m-1$. Otherwise, if  $\frac {2+\a}4 \theta_i  \in \mathbb N$ for some index $i$, then  
 \begin{equation} 	\label{formulaHenonIntro2} 
 	\textsf{m}(u_{p}) - \left(m + 2 \left\lceil \frac{\a}2\right\rceil + 2 \sum\limits_{i=1}^{m-1} \left[ \frac{2\!+\!\a }4 \theta_i \right]  \right) \in \Bigg[  - 2 \# \left\{ i = 1, \dots m-1 \, \Big| \, \frac {2\!+\!\a}4 \theta_i  \in \mathbb N\right\}   , \, 0    \Bigg],
 	\end{equation}
 	where $[\cdot]$ is the integer part and $\lceil \cdot \rceil$ the ceiling function.
In particular when $\alpha=0$ \eqref{formulaHenonIntro1} holds and  it reduces to
\begin{equation}
\label{formulaLaneEmdenIntro}
\textsf{m}(u_{p}) = 4m^{2}-m-2, \quad \forall\ p\geq \bar p.
\end{equation}
\end{theorem}

$\;$\\When $m=2$ Theorem \ref{theorem_Main_Intro} gives back \eqref{2zone}.

\

Observe that, since $\theta_{i}>2$ for every $i\geq 1$ (see \eqref{introtheta}), it follows that each  value given by \eqref{formulaHenonIntro1}  is strictly higher than the corresponding  value in the higher dimensional case given in formulas \eqref{LaneEmdenFormula}-\eqref{HenonFormulapcrit}, and hence also higher than the Morse index lower bound in \eqref{MorseLoweBound}. We stress that  in dimension 2, and for $\a>0$, the bound \eqref{MorseLoweBound} has been recently improved in \cite{dSdS19}, by exploiting the monotonicity of the Morse index with respect to the parameter $\a$. It is not difficult to check that for  $\a>0$ the value in \eqref{formulaHenonIntro1} is in general also strictly  higher than the corresponding value obtained in \cite{dSdS19}  (see Remark \ref{Ederson}).

\

Formula \eqref{formulaHenonIntro1} exhibits two kinds of discontinuity w.r.t. the parameter $\a$: one, occurring when $\a$ is an even integer, is a common phenomenon also with the higher dimensional case (\cite{AG_N3}); the other, occurring along the sequences $\a_{i,n} = 4n/\theta_i-2$,  is instead peculiar of dimension $2$ .

\
 
The interest in Theorem \ref{theorem_Main_Intro} is not just theoretical: the exact knowledge of the Morse index can be used in order to get multiplicity results for \eqref{problemaHenonIntro}, thus clarifying the structure of the set of its solutions. 
\\
This can be obtained for instance both via nonradial bifurcation from radial solutions associated to a change in the Morse index, and via minimization procedures in suitable symmetric settings combined with Morse index comparisons. These approaches have been explored in dimension  $N\ge3$ for all the radial solutions, while in dimension $N=2$ only the case of the radial  solution with  $m=2$ nodal zones  has been investigated so far (see \cite{GI20, AG-N2, Ama20p1, Ama20bif}).  Nevertheless   
there are numerical evidences that similar phenomena hold also when considering  radial solutions with more than $2$ nodal zones in dimension $N=2$ (see \cite{FazekasPacellaPlum_ultimo}), 
 and  in a  subsequent paper we plan to exploit the results in Theorem  \ref{theorem_Main_Intro} to treat this case.

 \

The proof of Theorem \ref{theorem_Main_Intro} follows a similar strategy to the one developed to get \eqref{2zone}  and \eqref{LaneEmdenFormula}-\eqref{HenonFormulapcrit}:  thanks to the change of variable \eqref{changeIntro} we can reduce to consider the Lane-Emden case ($\alpha=0$); then, after a spectral decomposition of an auxiliary singular eigenvalue problem associated to the linearized operator, we are finally lead to study the negative eigenvalues $\nu$ of the following  radial singular problem
\begin{equation}\label{singEigenvalueINTRO}
\begin{cases}
-(r \, \psi' )'= r \left( p| u_{p}|^{p-1} + \frac{\nu}{r^2} \right) \psi & \text{ as } 0 < r < 1 , \\
\psi=0 & \mbox{ if }r=1, \end{cases}
\end{equation}
where $u_{p}$ is the radial solution  of \eqref{problemaHenonIntro} (with $\alpha=0$) with $m$ nodal zones (see Section \ref{subsec:Morse} for more  details).
It is possible to show that \emph{negative} eigenvalues for problem \eqref{singEigenvalueINTRO} may be defined and are simple (\cite{GGN16}), moreover they are exactly $m$ which we denote by $\nu_{j}$, $j=1,\ldots, m$. The eigenvalues $\nu_{j}$ (and eigenfunctions $ \psi_{j}$) of \eqref{singEigenvalueINTRO}  obviously depend on  $u_{p}$, the core of the proof of Theorem \ref{theorem_Main_Intro} is thus the investigation of their asymptotic behavior   as $p\rightarrow +\infty$. 
We prove that 
\begin{theorem}
	\label{INTROthm:autovaloriAutofunzioni}
	For any $j=1,\ldots, m$
	\begin{equation}\label{limj}
	\lim_{p\rightarrow +\infty} \nu_{j}(p) = - \left(\frac{\theta_{m-j}}2\right)^2,
	\end{equation}
	where $(\theta_i)_{i\in\mathbb N}$ is the sequence in \eqref{introtheta}.
\end{theorem}

 Theorem \ref{INTROthm:autovaloriAutofunzioni} is part of a more general result which describes also the asymptotic behavior of the eigenfunctions (see Theorem \ref{thm:autovaloriAutofunzioni} for the complete statement). Its proof  is quite technical and, as already mentioned, it strongly relies on the \emph{tower of bubbles} asymptotic behavior of the radial solution $u_{p}$ as $p\rightarrow + \infty$ described very recently in \cite{IS-arxiv}, for any fixed number $m\geq 1$ of nodal zones (see Section \ref{section:asintoticaAutovalori}, see also  \cite{GGP14, AG-N2} for the case  $m=2$).\\
 The main difficulty, which is peculiar of the two dimensional case, is to understand the interaction between the different bubbles composing the profile of $u_p$ and the eigenfunctions of \eqref{singEigenvalueINTRO}.

  We shall see that each eigenfunction $\psi_{j}$ is \emph{synchronized} with a different bubble: precisely the first eigenfunction $\psi_{1}$ matches with the more external nodal zone of $u_p$ where the last bubble $Z_{0,m-1}$ appears, the second eigenfunction $\psi_{2}$ matches with the penultimate  bubble $Z_{0,m-2}$ and so on, till the last eigenfunction $\psi_{m}$ that matches with the first bubble $Z_{0,0}$ (see Section \ref{section:proofEigenvalues}). 
\\
Indeed, in the case $\alpha=0$, one can decompose formula \eqref{formulaHenonIntro1} as follows 
 	\begin{equation}\label{morsedecomp}
  	\textsf{m}(u_{p}) = \sum \limits_{i=1}^{m-1} \left(1+ 2 \left[\frac{\theta_{m-i}}{2}  \right]  \right) \ +\ 1, \end{equation}
 where
	each term 	
	``$ 1+ 2 \left[\frac{\theta_{m-i}}{2}  \right]  $'', coming from the $i^{th}$ eigenvalue of \eqref{singEigenvalueINTRO}, describes the contribution to the Morse index due to the  bubble $Z_{0, m-i}$, and the last term ``$1$'', coming from the $m^{th}$ eigenvalue, is due to the first bubble $Z_{0,0}$.
Observe that the  Morse index of each bubble (as a solution to \eqref{singLiouvilleINTRO} for $\alpha=0$) is known (see \cite{ChenLin})
 and coincides with the previous values:
\[\textsf{m}(Z_{0,0})=1 \qquad \mbox{ and }\qquad \textsf{m}(Z_{0, m-i})=1+ 2 \left[\frac{\theta_{m-i}}{2}  \right],  \]
so that \eqref{morsedecomp} may be rewritten as		
\[	\textsf{m}(u_{p})= \sum \limits_{i=1}^{m}	\textsf{m}(Z_{0, m-i}).\]
Moreover, one can explicitly compute (cfr. \cite{IS-arxiv}) the different contribution coming from each bubble
\[\textsf{m}(Z_{0, m-i})=1+ 2 \left[\frac{\theta_{m-i}}{2}  \right]=8(m-i)+3,\]
from which  formula \eqref{formulaLaneEmdenIntro} follows, which is \emph{nonlinear (quadratic)} in the number $m$ of nodal zones. We stress that in dimension $N\geq 3$ and for $\alpha=0$ formula \eqref{LaneEmdenFormula} holds, which is instead \emph{linear} in $m$. We notice that, since in this case the profile of the bubbles is given always by the same function $U_{0}$ (in \eqref{bubbleN} with $\alpha=0$) and it is known that  $\textsf{m}(U_{0})=1$, formula \eqref{LaneEmdenFormula} may be read as
\[\textsf{m}(u_{p})=m+(m-1)N= \textsf{m}(U_{0})+(N+1)\sum_{i=1}^{m-1}\textsf{m}(U_{0}).\]
\\

The paper is organized as follows: 
\tableofcontents

 \

\section{Asymptotic results for the Lane-Emden problem} \label{Section:Notazioni&Idee}
This section collects known results about the asymptotic behavior of the radial solutions in the case $\alpha=0$. Hence we consider the Dirichlet Lane-Emden problem
\begin{equation}\label{LE}
\left\{\begin{array}{ll}
-\Delta u = |u|^{p-1} u \qquad & \text{ in } B, \\
u= 0 & \text{ on } \partial B\end{array} \right.
\end{equation}
where $p>1$ and $B$ stands for the unit disk.

\

For any  $p>1$ and any $m\in\mathbb N$, $m\geq 1$, there exists a unique (up to a sign) radial solution   to \eqref{LE}  with exactly $m-1$ interior zeros  (see for instance \cite[p. 263]{Kajikiya}). 
\\
The solutions do not vanish in the origin and we denote by $u_p$ the unique nodal radial solution of \eqref{LE} having $m-1$ interior zeros and satisfying 
\[u_p(0)>0.\]

\

With a slight abuse of notation, we often  write  $u_p(r)=u_p(|x|)$.

\

\subsection{Asymptotic analysis of radial solutions}

Let us denote by $r_{i,p}$ the nodal radii of $u_p$ and by $s_{i,p}$ the critical radii of $u_{p}$ respectively, then it is known that
\[ 0=s_{0,p}<r_{1,p}<s_{1,p}<r_{2,p}<\ldots <r_{m-1,p}<s_{m-1,p}<r_{m,p}=1.
\]


Let us define the scaling parameters
\begin{equation}\label{epsilon}
\e_{i,p}  =\left(p |u_{p}(s_{i,p})|^{p-1}\right)^{-\frac{1}{2}}, \qquad i=0,\ldots, m-1,
\end{equation}
and rescale  the solutions in each nodal zone as
\begin{equation}\label{resc-sol}
 u^i_{p}(r) : = p\frac{u_p(\e_{i,p} x) - u_p(s_{i,p})}{u_p(s_{i,p})} \quad \text{ as }  
r\in\left\{\begin{array}{lr}
[0,\frac{r_{1,p}}{\varepsilon_{0,p}}], \qquad \quad \qquad \ \mbox{ if }i=0,
\\
\left[\frac{r_{i,p}}{\varepsilon_{i,p}},\frac{r_{i+1,p}}{\varepsilon_{i,p}}
\right],\qquad \mbox{ if }i= 1,\ldots, m-1.
\end{array}
\right.
\end{equation}

\

Let $(\theta_i)_i$ be the sequence  defined in \eqref{introtheta}, which satisfies (see \cite{IS-arxiv}):
\begin{equation}\label{useful2}
\theta_0=2,\quad
	  8i+2 < \theta_i <8i +4,\ \forall i\geq 1.
	\end{equation} 
We also introduce	
\begin{equation}\label{bubblecomplete}
Z_{i}(x):=\log\frac{2\theta_i^2\gamma_i|x|^{(\theta_i-2)}}{(\gamma_i+|x|^{\theta_i})^2},
\quad\mbox{ where }\quad\gamma_i:=	\frac{\theta_{i}+2}{\theta_{i}-2} \left(\frac{\theta_{i}^2-4}{2}\right)^{\frac{\theta_{i}}{2}}.\end{equation}

Observe that the function $Z_{i}$ is a radial solution of
\begin{equation}
\label{prob-lim-i}
\left\{
\begin{array}{lr}
-\Delta Z_{i}=e^{Z_{i}}+ 2\pi(2-\theta_i)\delta_{0}\quad\mbox{ in }\R^2,\\
Z_{i}(\sqrt{\frac{\theta_i^2-4}{2}})=0,\\
\int_{\R^2}e^{Z_{i}}dx=\frac{8\pi\theta_i}{2},
\end{array}
\right.
\end{equation}
where  $\delta_{0}$ is the Dirac measure centered at $0$. 
In particular in the case  $i=0$, since the constant $\theta_{0}=2$, $Z_{0}$ solves the standard Liouville equation
\begin{equation}\label{prob-lim-0}
\begin{cases}
-\Delta Z_{0} = e^{Z_{0}} &  x \in \R^2 , \\
Z_{0}(0)=0&\\
\int_{\R^2} e^{Z_{0}} dx =8\pi.&\end{cases}
\end{equation}

From \cite[Theorem 2.5]{IS-arxiv} we know that $u_p$ has a \emph{tower of bubbles} behavior in the limit as $p\rightarrow +\infty$, with bubbles  given  by the functions $Z_i$, $i=0,\ldots, m-1$:
\begin{lemma}[\cite{IS-arxiv}]\label{teo:asympt} 
 As $p\to \infty$ we have
	\begin{equation}
	\label{zeri-lim}	
	\frac{r_{i,p}}{\e_{i,p}} \rightarrow 0  \ (i\neq 0), \qquad \frac{r_{i+1,p}}{\e_{i,p}} \to \infty , \qquad \frac{s_{i,p}}{\e_{i,p}}  \rightarrow \sqrt{\frac{\theta_i^2-4}{2}}, 
	\end{equation}
	for $i=0,\dots m-1$. Furthermore
	\begin{eqnarray} 
	\label{sol-lim-0} && u^0_{ p} \longrightarrow  Z_{0}\ \text{ in } C^1_{\loc}(\R^2), \\
	\label{sol-lim-i} &&  \tuip  \longrightarrow Z_{i}\    \text{ in } C^1_{\loc}(\R^2\setminus\{0\}),  \mbox{ for $i=1,\dots m-1.$} 
	\end{eqnarray} 
\end{lemma}

\

Last we recall some pointwise estimates that will be useful in the study of the linearized operator at $u_{p}$. Let $f_{p}$ be the following function
\begin{equation}\label{def:f}
f_p(r):=p\, r^2|u_p(r)|^{p-1}, \quad   0\leq r\leq1
\end{equation}
and for any $K>1$ and $p>1$ let us  define the set $G_{p}(K)\subset [0,1]$ as
\begin{equation}\label{def:G} G_{p}(K):=\bigcup_{i=0}^{m-2}[K\varepsilon_{i,p},\frac{1}{K}\varepsilon_{i+1,p}]\cup [K\varepsilon_{m-1,p},1].\end{equation}
In  \cite[Proposition 6.10]{DIP-N2} it has been proven that
\begin{lemma}\label{lemma-stima-fp} 
	There exists $C>0$ such that
	\begin{equation}\label{stima-uno}
	f_p(r) \leq C \ \text{ for any } r\ge 0 \text{ and } p>1.
	\end{equation}
	Moreover for any $\delta>0$ there exist $K(\delta) >1$ and $p(\delta)>1$ such that for any $K>K(\delta)$ and $p\geq p(\delta)$
	\begin{equation}\label{stima-due} 
	\max \left\{ f_p(r) \, : \,  r\in  G_{p}(K) \right\} \leq \d .
	\end{equation}
\end{lemma}

\

\

\section{Strategy for the Morse index computation}\label{subsec:Morse}

We will first consider the Lane-Emden problem ($\alpha=0$) and prove Theorem \ref{theorem_Main_Intro} in this case (see Section \ref{proofMorseLaneEmden}), finally in Section \ref{sec:Henon} we will treat the H\'enon problem ($\alpha>0$)  by exploiting  the change of variable \eqref{changeIntro} and prove Theorem \ref{theorem_Main_Intro} in its full generality.

 \
 
 This section describes the strategy that we will adopt in order to compute the Morse index in  the case $\alpha=0$. More precisely we will show how the computation of the Morse index may be reduced to the study of the \emph{size} of the negative radial eigenvalues of a suitable singular eigenvalue problem (see formula \eqref{morse-index-formula} below). The study of these eigenvalues and the conclusion of the  proof of Theorem \ref{theorem_Main_Intro} (in the case $\alpha=0$) is instead the goal of Sections \ref{section:asintoticaAutovalori} and  \ref{proofMorseLaneEmden}, respectively.
 
 \

As before we denote by $u_p$ the radial solution to the Lane-Emden problem \eqref{LE} having $m-1$ interior zeros and keep all the notations introduced in Section  \ref{Section:Notazioni&Idee}.
 
 \
 
As already recalled the \emph{Morse index} of $u_p$\remove{, that we denote hereafter by $\textsf{m}(u_p)$,} is the maximal dimension of a subspace  of $ H^1_0(B)$ in which the quadratic form 
\begin{equation}
\label{forma-quadratica}
{\mathcal Q}_p(\phi) =\int_B \left(|\nabla \phi|^2 -V_p(x) \phi^2\right) dx
\end{equation}
is negative defined, where \begin{equation}
\label{V_p}
V_p(x)  := p| u_p(x)|^{p-1} .
\end{equation}
Since $u_{p}$ is a radial solution we can also consider the \emph{radial Morse index} of $u_{p}$, denoted by $\textsf{m}_{\mathrm{rad}}(u_p)$, which is the maximal dimension of a subspace $X$ of $ H^1_{0,\mathrm{rad}}(B)$ (the subspace of radial functions in $ H^1_0(B)$) such that ${\mathcal Q}_p(\phi)<0$, $\forall \phi\in X\setminus\{0\}$.

\

Observe that $B$ is a bounded domain, so $\textsf{m}(u_p)$ (resp. $\textsf{m}_{\mathrm{rad}}(u_p)$) coincides with the number of the negative eigenvalues (resp. radial eigenvalues) $\Lambda(p)$,  counted with multiplicity, of  the linearized operator $L_{p}:-\Delta-V_{p}(x)$ at $u_{p}$, i.e.:
\begin{equation}\label{standard-eig-prob}
-\Delta \phi-V_p(x) \phi= \Lambda(p) \phi , \qquad \phi \in H^1_0(B) \mbox{ (resp. $\phi \in H^1_{0,\mathrm{rad}}(B) $)}.
\end{equation}

\

It is well known (see \cite{HRS11, BW03}) that 
\begin{equation}\label{radialem}
\textsf{m}_{\mathrm{rad}}(u_p)=m,
\end{equation}
where $m$ is the number of nodal zones of $u_{p}$, moreover $u_{p}$ is radially non-degenerate (see for instance \cite{GI20}).

\

In order to computer $\textsf{m}(u_p)$ we follow the same general strategy already used in \cite{DIP-N>3, DIP-N2, AG_N3, AG-N2, GI20}: instead of counting the negative eigenvalues of \eqref{standard-eig-prob}, we consider an auxiliary singular eigenvalue problem which allow to exploit a spectral decomposition and hence to reduce to a  radial eigenvalue problem.

\subsection{Singular eigenvalue problem and spectral decomposition}

\

\

It is possible to show that $\textsf{m}(u_p)$ coincides with  the number of negative eigenvalues  $\widehat\Lambda(p)$, counted with multiplicity,  of the following auxiliary eigenvalue problem associated to the linearized operator $L_p $:
\begin{equation}\label{singular-eig-prob}
-\Delta \phi -V_p(x) \phi = \widehat\Lambda(p) \frac{\phi}{|x|^2}  , \qquad \phi \in {\mathcal H}_0 ,
\end{equation} 
in the  weighted Sobolev space
\[
\mathcal{H}_0 =\mathcal L\ \cap\ H_0^1(B), \ \mbox{ where } {\mathcal L}  = \{\phi:  B\to \R\, : \,  \phi/|x|  \in L^2(B)\}.
\] 

This equivalence is quite straightforward in the case of domains which do not contain the origin (see for instance  \cite{GGPS11},  where it is proved in the case when the domain is an annulus).  In our case, since $0\in B$,   \eqref{singular-eig-prob}  is a singular problem. Nevertheless its \emph{negative} eigenvalues may be variationally characterized  despite a \emph{lack of compactness} (see \cite{GGN16}, for more details see also \cite[Section 3.2]{GI20}, and  \cite{AG-sing1} for a more general setting) and the equivalence between the number of the negative eigenvalues of \eqref{standard-eig-prob} and \eqref{singular-eig-prob} can be proved (see \cite[Lemma 2.6]{GGN16}, see also \cite[Lemma 3.5]{GI20}, \cite[Proposition 1.1]{AG-sing1}).
%
%
%
%

\

The main advantage of dealing with the singular problem \eqref{singular-eig-prob} instead of \eqref{standard-eig-prob} is that the eigenfunctions of \eqref{singular-eig-prob} can be easily projected along the spherical harmonics. This implies a 
spectral decomposition for the eigenvalues $\widehat \L (p)$ of \eqref{singular-eig-prob} into a radial and an angular part:
\begin{equation}\label{decomposition}
\widehat \L (p)=k^2 + \nu(p),\end{equation}
 where $k^{2}$, for $k=0,1,2, \ldots$ are the eigenvalues of the Laplace-Beltrami operator $-\Delta_{\mathbb S^{1}}$ (the angular part) and
$\nu(p)$ are the (negative) radial eigenvalues of \eqref{singular-eig-prob}, namely they satisfy the following singular Sturm-Liouville problem 
\begin{equation}
\label{S-L}
-(r \, \psi' )'= r \left( V_p(r) + \frac{\nu(p)}{r^2} \right) \psi, \qquad
\psi\in {\mathcal H}_{0,\mathrm{rad}}=\mathcal L\ \cap\ H_{0,\mathrm{rad}}^1(B).
\end{equation}
We stress that  $\widehat \L (p)$ is negative iff
\begin{equation}\label{babel}\sqrt{-\nu(p)}>k.\end{equation}
Hence in order to compute $\textsf{m}(u_p)$ one \emph{reduces to study \eqref{babel} for the negative eigenvalues $\nu(p)$} of the $1$-dimensional problem \eqref{S-L}.\\ For more details about the spectral decomposition the reader may look at \cite{Lin95, GGPS11, GGN16}, or to the more recent \cite[Lemma 3.7]{GI20},  \cite[Section 4]{AG-sing1}.

\

\subsection{Variational characterization of the negative eigenvalues and eigenfunctions of \eqref{S-L}}\label{subsec:var-car}

\

\

As already said,  the \emph{negative} eigenvalues for problem \eqref{S-L} may be defined variationally despite the singularity of the Sturm-Liouville problem \eqref{S-L} at the origin, moreover they are simple  and  by \eqref{radialem} we know that they are exactly $m$, which we denote by $\nu_{j}(p)$, $j=1,\ldots, m$.\\
Here we recall their variational characterization and the definition of the corresponding eigenfunctions (cfr. \cite{GGN16}, see also \cite[Section 3]{AG-sing1}):
\begin{equation}\label{var-char-1}
	\nu_1(p) := \min\left\{ \frac{\int_0^1 r\left(|\psi'|^2 - V_p \psi^2\right)dr}{\int_0^1 r^{-1} \psi^2 dr }: \psi\in \mathcal H_{0,\mathrm{rad}} , \ \psi\neq 0 \right\};
	\end{equation}
	since it is negative,  it can be proven that it is attained by a function $\psi_{1,p}\in \mathcal H_{0,\mathrm{rad}}$ which solves \eqref{S-L} in a weak sense, and  which is  therefore called an \emph{eigenfunction} related to the eigenvalue $\nu_1(p)$; w.l.g. we may assume that it is normalized in $\mathcal L$, i.e.  
	$\int_{0}^{1}r^{-1}(\psi_{1,p})^{2}=1$. Iteratively, for $j=2,\ldots, m$, one has	\begin{equation}\label{var-char} 
	\nu_{j}(p) :=  \min\left\{ \frac{\int_0^1r \left(|\psi'|^2 - V_p \psi^2\right) dr}{\int_0^1r^{-1} \psi^2 dr} \, : \, \psi\in \mathcal H_{0,\mathrm{rad}}, \ \psi \underline{\perp} \psi_{1,p},\dots \psi_{j-1,p} \right\},
	\end{equation}
	where the symbol $\underline{\perp}$ denotes orthogonality in $\mathcal L$, i.e.
	\[ \varphi \underline\perp \psi \Longleftrightarrow \int_0^1 r^{-1} \varphi\psi dr = 0,\]
	Again, since $\nu_{j}(p) <0$ for any $j=2,\ldots, m$, then the infimum is attained by an eigenfunction $\psi_{j,p}$, which solves \eqref{S-L}  in a weak sense and that w.l.g. satisfies
\begin{equation}\label{normalization}
\int_0^1 r^{-1} \psi_{j,p} \psi_{h,p}  dr = \delta_{j h}.
\end{equation}

\

 Furthermore one can prove that the eigenvalues are simple and that (see \cite[Proposition 3.3, Theorem 1.3]{AG-sing2})
\begin{equation}\label{est-rse}
\nu_{1}(p) < \nu_2(p)<\dots \nu_{m-1}(p)< -1 < \nu_m(p) <0,  
	\end{equation}
for any $p>1$.

\

\subsection{Computation of $\textsf{m}(u_p)$ by the size of the negative eigenvalues of \eqref{S-L}}\label{subsec:comp-morse}

\

\

By \eqref{decomposition} and \eqref{babel}, and recalling  that the eigenvalues $\nu_{j}(p)$ defined in \eqref{var-char-1}-\eqref{var-char} are simple while the eigenvalues $k^2$ of the Laplace-Beltrami operator $-\Delta_{\mathbb S^{1}}$ have multiplicity $1$ if $k=0$ and $2$ when $k\geq 1$,  it follows that
\begin{equation}\label{morse-index-formula}
\textsf{m}(u_p)=  m + 2 \sum\limits_{j=1}^{ m-1} \left\lceil \sqrt{-\nu_j(p)} -1 \right\rceil  ,
\end{equation}
for any $p>1$. 
\

\

\section{Asymptotic behavior of $\nu_{j}(p)$ as $p\rightarrow+\infty$}
\label{section:asintoticaAutovalori}

In this section we study the asymptotic behavior, as $p\rightarrow +\infty$, of the singular eigenvalues $\nu_{j}(p)$, $j=1,\ldots, m$, defined in \eqref{var-char-1}-\eqref{var-char}.

\

In order to compute their limit values we will properly scale the corresponding eigenfunctions $\psi_{j,p}$ according  to each scaling parameter $\e_{i,p}$ introduced in \eqref{epsilon} and then pass to the limit into the equations satisfied by the rescaled functions. This will be possible thanks to the asymptotic results on the solutions $u_{p}$ of the Lane-Emden problem \eqref{LE} collected in Section \ref{Section:Notazioni&Idee}. Furthermore we will analyze the limit eigenvalue problems obtained (see Lemma \ref{lemma:problemaLimiteAutovalori} below).

\

Our results about the asymptotic behavior of the eigenvalues and the rescaled eigenfunctions are stated in Theorem \ref{thm:autovaloriAutofunzioni} below (which is the complete version of Theorem \ref{INTROthm:autovaloriAutofunzioni}  in Section \ref{section:intro}).

\

Next we introduce some notation and observations needed to state Theorem \ref{thm:autovaloriAutofunzioni}. 

\

We denote by $\psi^i_{j,p}$, for $i=0, \dots m-1$, the  $m$ functions obtained from rescaling each eigenfunction $\psi_{j,p}$ as follows:
\begin{equation}\label{autof-resc}
\psi^i_{j,p}(r) := \begin{cases} \psi_{j,p}(\e_{i,p} r) & \quad \text{ in }   \left[0, \frac{1}{\e_{i,p}}\right) 
\\
0 &  \quad \text{ otherwise.} \end{cases} 
\end{equation}
Observe that $\psi^i_{j,p}$  belong to the closure of $C_0^{\infty}(0,\infty)$ with respect to the norm 
\[ \left(\int_0^{\infty} \left( r |\psi'|^2 + r^{-1}\psi^2 \right) dr \right)^{\frac{1}{2}},\]
which will be denoted by ${\mathcal D}_{\mathrm{rad}}$, and solve
\begin{equation}\label{S-L-resc}
-(r \, ( \psi^i_{j,p})' )'= r \left(  V^i_p + \frac{\nu_j(p)}{r^2} \right)  \psi^i_{j,p} \end{equation} 
in $[0,\frac{r_{1,p}}{\varepsilon_{0,p}}]$  if $i=0$,
in $\left[\frac{r_{i,p}}{\varepsilon_{i,p}},\frac{r_{i+1,p}}{\varepsilon_{i,p}}
\right]$ if $i= 1,\ldots, m-1$, 
with
\begin{equation}\label{pot-risc}
 V^i_p(r) :=  (\e_{i,p})^2V_p(\e_{i,p} r),
\end{equation}
where $V_{p}$ is defined in \eqref{V_p}.
Moreover by the definition \eqref{autof-resc} and the normalization \eqref{normalization}, we have 
\begin{equation}\label{norm1}
\int_0^{\infty}r^{-1}(\psi_{j,p}^i)^2dr\leq\int_0^1r^{-1}(\psi_{j,p})^2dr=1
\end{equation}
\begin{equation}\label{norm2}
\int_0^{\infty}r((\psi_{j,p}^i)')^2dr\leq\int_0^1r(\psi_{j,p}')^2dr.
\end{equation}

\

Thanks to Lemma \ref{teo:asympt} the set $\left[0, \frac{r_{1,p}}{\e_{0,p}}\right)$  invades $[0,\infty)$ in the limit as $p\rightarrow +\infty$, while the sets $\left(\frac{r_{i,p}}{\e_{i,p}}, \frac{r_{i\!+\!1,p}}{\e_{i,p}}\right)$, for $i=1,\dots m-1$,   invade $(0,\infty)$. Furthermore  \begin{eqnarray} 
\label{pot-lim-0} && V^0_{p}=\left(1+\frac{u_{p}^{0}}{p}\right)^{p-1} \longrightarrow  e^{Z_{0}} \ \text{ in } C^{0}_{\loc}[0,\infty), \\
\label{pot-lim-i}	&&  V^i_{p}=\left(1+\frac{u_{p}^{i}}{p}\right)^{p-1} \longrightarrow  e^{Z_{i}}  	\  \text{ in } C^{0}_{\loc}(0,\infty), \mbox{ for }i=1,\ldots, m-1,
\end{eqnarray}
 where $Z_i$ are the functions in \eqref{bubblecomplete}. Hence, if we prove that we can pass to the limit into equations \eqref{S-L-resc}, then the natural limit problems  will be the following eigenvalue problems 
\begin{equation}\label{S-L-lim}
\begin{cases}-(r \, (\eta)' )'= r \left( e^{Z^i} + \frac{\beta}{r^2} \right) \eta &  \text{ as }  \ r> 0 , \\
\eta \in {\mathcal D}_{\mathrm{rad}}, & \end{cases}
\end{equation}
for $i=0,\ldots, m-1$.
From \cite[Section 5]{DIP-N>3} and \cite[Section 5.2]{AG-N2} we know that \eqref{S-L-lim} admits only one negative eigenvalue, which can be explicitly characterized:
\begin{lemma}\label{lemma:problemaLimiteAutovalori}
Let $i\in \{0,\ldots, m-1\}$ and let $\beta$ be an eigenvalue to \eqref{S-L-lim}. Then
\begin{equation}\label{S-L-lim-autov}
\beta<0\quad \mbox{ iff }\quad\beta=\beta^i := - \left(\frac{\theta_i}{2}\right)^2,
\end{equation}
where $\theta_{i}$ is the number given by \eqref{introtheta}. Moreover in such a case the eigenvalue $\beta^{i}$ is simple and its eigenspace is  spanned by 
\begin{equation}\label{S-L-lim-autof}
\eta(r)=\eta^i(r):= \frac{\sqrt{\theta_i\gamma_i} \, r^{\frac{\theta_i}2}}{\gamma_i + r^{\theta_i}},  \quad\mbox{ where }\quad\gamma_i:=	\frac{\theta_{i}+2}{\theta_{i}-2} \left(\frac{\theta_{i}^2-4}{2}\right)^{\frac{\theta_{i}}{2}}.
\end{equation}
\end{lemma}
Notice that $\eta^i$ is normalized so that
\begin{equation}\label{eta-normalizziata}
\int_0^{\infty} r^{-1}(\eta^i)^2 dr =1 .
\end{equation}

As a consequence of Lemma \ref{lemma:problemaLimiteAutovalori} it follows that all the numbers $\beta=\beta^{i}$ in \eqref{S-L-lim-autov}, for $i=0,\ldots,m-1$,  are   candidates to be the limit value of each eigenvalue $\nu_{j}(p)$, as $p\rightarrow +\infty.$

We remark that, for $i=0,\ldots, m-1$, the limit problems \eqref{S-L-lim}, as well as their  negative eigenvalue $\beta^{i}$ in \eqref{S-L-lim-autov}, are different  from one another, in particular combining \eqref{S-L-lim-autov} and \eqref{useful2}  we know that the following strict order holds:
\begin{equation}\label{ordine-beta}
\beta^{m-1}< \dots \beta^1 <-25 < \beta^0= -1 .\end{equation}
In order to select the right limit value of $\nu_{j}(p)$ among all the $\beta^{i}$'s,
we need  thus to understand which one (if any) among the possible scalings  $\psi^i_{j,p}$, for $i=0,\ldots, m-1$, does not vanish as $p\to\infty$. 
\\

We shall see that
\begin{theorem}
	\label{thm:autovaloriAutofunzioni}
	For any $j=1,\ldots, m$ 
	\begin{equation}\label{limj}
	\lim_{p\rightarrow +\infty} \nu_{j}(p)=\beta^{m-j}  = - \left(\frac{\theta_{m-j}}2\right)^2
	\end{equation}
	Moreover there exists $A_{j}\neq 0$ such that 
	\begin{eqnarray*}
		&&\psi^{m-j}_{j,p} \to  A_{j}\eta^{m-j}\\
		&&\psi_{j,p}^{i}\to  0,\qquad \quad i=0,\ldots, m-1, \ i\neq m-j
	\end{eqnarray*}
	weakly in $\mathcal D_{\mathrm{rad}}$ and strongly in $C^1_{\loc}(0,\infty)$.
\end{theorem}

Observe that Theorem \ref{thm:autovaloriAutofunzioni} describes the asymptotic also for the last eigenvalue $\nu_{m}(p)$, 
even if this is not needed for the computation of the Morse index.

\

\subsection{The proof of Theorem \ref{thm:autovaloriAutofunzioni}} \label{section:proofEigenvalues}

The proof of Theorem \ref{thm:autovaloriAutofunzioni} is based on an iterative procedure on the index $j$.

First we prove the result for $j=1$:
\begin{proposition} \label{prop:BaseInduttiva}
	\begin{equation}\label{lim1}
	\lim_{p\rightarrow +\infty}\nu_{1}(p)=\beta^{m-1}
	\end{equation}
	Moreover there exists $A_{1}\neq 0$ such that 
	\begin{eqnarray*}
		&&\psi_{1}^{m-1} \to A_{1}\eta^{m-1}  \\
		&&\psi_{1}^{i}\to  0, \qquad  \qquad i=0,\ldots, m-2
	\end{eqnarray*}
\end{proposition}
Then we prove the inductive step
\begin{proposition}
	\label{prop:autovAutofIteration}
	Let $h\in\{2,\ldots , m-1\}$. Assume that Theorem \ref{thm:autovaloriAutofunzioni} holds true for any $j=1,\ldots, h-1$. Then it holds true for $j=h.$
\end{proposition}

The last eigenvalue has to be treated separately, namely we conclude proving
	
	\begin{proposition} \label{prop:autovAutofUltimo}
		\begin{equation}\label{limm}
		\lim_{p\rightarrow +\infty}\nu_{m}(p)=\beta^0=-1
		\end{equation}
		Moreover there exists $A_{m}\neq 0$ such that
		\begin{eqnarray*}
			&&\psi_{m}^{0} \to A_{m}\eta^{0}\\
			&&\psi_{m}^{i}\to 0, \quad \qquad i=1,2,\ldots, m-1 .
		\end{eqnarray*}
	\end{proposition}

\subsection{Preliminary convergence results} 
We start showing that the eigenvalues  $\nu_j(p)$ and the rescaled eigenfunctions  $\psi^i_{j,p}$ are uniformly bounded in $p$.
\begin{lemma}\label{lemma:bound}
	There exists $C>0$ such that for every $p>1$ we have
	\begin{equation}\label{A1}
	-C\leq\nu_1(p)<\nu_2(p)<\ldots<\nu_m(p)<0
	\end{equation}
	\begin{equation}\label{A2}
	\int_0^{\infty}r((\psi^i_{j,p})')^2dr\leq C
	\end{equation}
	for every $i=0,\ldots,m-1$ and $j=1,\ldots,m$.
\end{lemma}
\begin{proof}
	Using $\psi_{j,p}$ as a test function in \eqref{S-L} we get
	\begin{equation}\label{appoj}
	\int_0^1 r(\psi'_{j,p})^2dr=\int_0^1r(V_p+\frac{\nu_j(p)}{r^2})(\psi_{j,p})^2dr.
	\end{equation}
	For $j=1$, by virtue of \eqref{normalization} we can extract $\nu_1(p)$ getting that
	\[
	\nu_1(p)=\int_0^1r((\psi'_{1,p})^2-p|u_p|^{p-1}(\psi_{1,p})^2)dr\geq-\sup_{(0,1)}f_p(r)\int_0^1 r^{-1}(\psi_{1,p})^2dr=-C
	\]
	thanks to Lemma \ref{lemma-stima-fp} and \eqref{norm1}.\\
	Besides, since $\nu_j(p)<0$ for $j=1,\ldots,m$ by \eqref{est-rse}, \eqref{appoj}, \eqref{norm1} and Lemma  \ref{lemma-stima-fp}
	\[
	\int_0^1 r(\psi_{j,p}')^2dr<\int_0^1r^{-1}f_p(\psi_{j,p})^2dr\leq\sup_{r\in(0,1)}f_p(r)\int_0^1 r^{-1}(\psi_{j,p})^2dr=C.
	\]
	So also \eqref{A2} is proved, recalling \eqref{norm2}.
\end{proof}

As a consequence we can thus prove:
\begin{proposition}\label{lem:conv-autof-resc}
	Let $j=1,\ldots,m$.
	Then there exist a sequence $p_{n}\rightarrow+\infty$, a number $\bar\nu_j\leq0$ and  $m$ functions $\overline\psi^i_j$, for $i=0,\ldots,m-1$,  such that as $n\rightarrow +\infty$
	\begin{align} \label{lim-autov}
	\nu_j(p_{n})\to\bar\nu_j & \\
	\label{lim-autof-risc}
	\psi_{j,p_{n}}^i
	\to\overline\psi^i_j & \quad\text{weakly in $\mathcal D_{\mathrm{rad}}$ and strongly in $L^2_{\loc}(0,\infty)$}.
	\end{align}
	Moreover $\overline\psi^i_j$  is a weak solution to 
	\eqref{S-L-lim} with eigenvalue $\beta=\bar\nu_{j}$.	
\end{proposition}
\begin{proof}
	By \eqref{A1} we can extract a sequence $p_{n}\rightarrow +\infty$ such that $\nu_{j}(p_n)\to\bar\nu_j\leq0$. \eqref{norm1} and \eqref{A2} imply that the sequence $(\psi_{j,p_n}^i)_{n}$ is uniformly bounded in $\mathcal{D}_{\mathrm{rad}}$ hence, up to another subsequence (that we still denote by $p_{n}$), one has that $
	\psi_{j,p_n}^i\to\overline\psi^i_j$ weakly in $\mathcal{D}_{\mathrm{rad}}$, strongly in $L^2_{\loc}(0,\infty)$ and almost everywhere in $(0,\infty)$.
	In particular $\overline\psi^i_j\in\mathcal D_{\mathrm{rad}}$. Since by {\eqref{zeri-lim}}  the intervals 
	\begin{equation*}
	 I^i_{p}:=\left\{
\begin{array}{ll}
(0,\frac{r_{1,p}}{\varepsilon_{0,p}})&\text{if $i=0$}\\
(\frac{r_{i,p}}{\varepsilon_{i,p}},\frac{r_{{i+1},p}}{\varepsilon_{i,p}})&\text{if $i>0$}
\end{array}
	\right.
	\end{equation*} 
	invade $(0,\infty)$, as $p\to+\infty$, for every $\varphi\in C_0^\infty(0,\infty)$ we can choose $n$ so large in such a way that $\mathrm{supp}\varphi\subset I^i_{p_n}$ and $\psi_{j,p_n}^i$ verifies 

	\[
	\int_0^\infty r (\psi_{j,p_n}^i)'\varphi'dr=\int_0^\infty r  V^i_{p_n}\psi_{j,p_n}^i\varphi dr+\nu_j(p_n)\int_0^\infty r^{-1} (\psi_{j,p_n}^i)\varphi dr.
	\]
	The weak convergence in $\mathcal D_{\mathrm{rad}}$ then implies that 
	\[
	\int_0^\infty r  (\psi_{j,p_n}^i)'\varphi'dr\to\int_0^\infty r(\overline\psi_j^i)'\varphi'dr
	\]
	\[
	\int_0^\infty r  \psi_{j,p_n}^i\varphi dr\to\int_0^\infty r\overline\psi_j^i \varphi dr
	\]
	while the strong convergence in $L^2_{\loc}$ and the fact that $ V^i_{p_n}\to e^{Z_i}$ in $C^1_{\loc}(0,\infty)$ imply also that
	\[
	\int_0^\infty r {V}^i_{p_n}\psi_{j,p_n}^i\varphi dr\to\int_0^\infty r e^{Z_i}\overline\psi^i_j\varphi dr
	\]
	getting that $\overline\psi^i_j$ solves \eqref{S-L-lim} with $\beta=\bar\nu_j$ in the weak sense.
	
\end{proof}

Thanks to Lemma \ref{lemma:problemaLimiteAutovalori}, we can deduce some crucial consequences of Proposition \ref{lem:conv-autof-resc}
\begin{corollary} \label{remark} 
Let $\bar \nu_j$ and $\overline\psi^i_j$ be as in Proposition \ref{lem:conv-autof-resc} and $\beta^{i}, \eta^{i}$  as in \eqref{S-L-lim-autov} and \eqref{S-L-lim-autof}. It holds
\begin{enumerate}
		\item[$(i)$] If $\bar \nu_j \neq \beta^i, 0$, then $\overline\psi^i_j\equiv 0$.
		\item[$(ii)$] If there exists $j\in\{1, \dots, m-1\}$ such that $\overline\psi^i_j \not\equiv 0$, then $\bar \nu_j= \beta^i$.\\ 
		Furthermore 
		\begin{eqnarray}\label{caratAutofunzNonnulla}
		&&\overline\psi^i_j= A_j \eta^i\mbox{ for some }A_j\neq 0, \ |A_j|\leq1
		\\
		&&\label{caratAutofunzNulla} \overline\psi^h_j \equiv 0 \mbox{ for every }h\neq i.
		\end{eqnarray}
\end{enumerate}
\end{corollary}
\begin{proof} $(i)$ is a direct consequence of  Proposition \ref{lem:conv-autof-resc} and Lemma \ref{lemma:problemaLimiteAutovalori}. Indeed $\bar\nu_{j}\leq 0$ is an eigenvalue of problem \eqref{S-L-lim} and the only strictly negative eigenvalue must be $\beta^i$.

The first assertion of $(ii)$ follows from $(i)$, observing also that, thanks to \eqref{est-rse}, $\bar\nu_{j}\le -1$ as $j=1, \dots, m-1$ (while $-1\le \bar \nu_m \le 0$). Then Lemma \ref{lemma:problemaLimiteAutovalori} implies  that $\overline\psi^{i}_{j}=A_{j}\eta^{i}$, for a certain $A_{j}\in\mathbb R$. As a consequence, by the convergence in \eqref{lim-autof-risc} and Fatou's Lemma, one deduces that
\[(A_{j})^{2}\overset{\eqref{eta-normalizziata}}{=}(A_{j})^{2}\int_{0}^{\infty}r^{-1} (\eta^i)^2=\int_{0}^{\infty}r^{-1} (\overline\psi_{j}^i)^2\leq \liminf_{p\rightarrow +\infty}\int_{0}^{\infty}r^{-1} (\psi_{j,p}^i)^2\overset{\eqref{norm1}}{\leq} 1,\]
which implies \eqref{caratAutofunzNonnulla}.
Finally $0\neq\bar\nu_{j}=\beta^{i}\neq\beta^{h}$, for $h\neq i$,  by \eqref{ordine-beta}, hence \eqref{caratAutofunzNulla}  follows from $(i)$.\end{proof}

The convergence in \eqref{lim-autof-risc} is actually stronger, as stated by the following Lemma.

\begin{lemma}\label{lemma:C1loc} 
	Using the same notation of Proposition \ref{lem:conv-autof-resc}, we have
	\begin{equation} 
	\label{lim-autof-risc-unif}
	\psi_{j,p_n}^i \to\overline\psi^i_j  \ \text{ strongly in }C^1_{\loc}(0,\infty) , 
	\end{equation}
as $n\rightarrow +\infty$, for  $j=1,\ldots,m$, $i=0,\ldots,m-1$.\\	
Furthermore, if  $\bar \nu_j\leq-25$, then   
	\begin{equation}\label{lim-autof-risc-0} 
	\psi_{j,p_n}^0   \to\overline\psi^0_j  \ \text{ in $C^1_{\loc}[0,\infty)$,}
	\end{equation}
as $n\rightarrow +\infty$, for  $j=1,\ldots,m$.\end{lemma}
\begin{proof}
	Recall tha $	\psi_{j,p_n}\in\mathcal{H}_{0,\mathrm{rad}}\subset C^0(0,1]$ and  $\psi_{j,p_n}$ is a solution to \eqref{S-L} (with $ V_{p_n} \in C^\infty[0,1]$), so  $\psi_{j,p_n}\in C^1(0,1]$ and in turn via a bootstrap argument  $\psi_{j,p_n}\in C^\infty(0,1]$.
	If $r_2\geq r_1\geq R^{-1}>0$ we have
	\[
	|\psi_{j,p_n}^i(r_2)-\psi_{j,p_n}^i(r_1)|\leq \int_{r_1}^{r_2}|(\psi_{j,p_n}^i)'(t)|dt\overset{\eqref{A2}}{\leq} C\left(\int_{r_1}^{r_2}t^{-1}dt\right)^{\frac12}\leq CR^{\frac12}\sqrt{r_2-r_1}
	\]
	so (up to another subsequence) $\psi_{j,p_n}^i\to\overline\psi^i_j$ uniformly in any set of type $[R^{-1},R]$ by the Arzelà-Ascoli Theorem. Furthermore,  by  equation \eqref{S-L-resc}, it is easy to derive a  bound for $\psi_{j,p_n}^i$ in $C^2(R^{-1},R)$, which ensures the   convergence in $C^1(R^{-1},R)$, completing the proof of \eqref{lim-autof-risc-unif}.\\
	Next we derive \eqref{lim-autof-risc-0}.

	
	Reasoning as in \cite[Lemma 2.4]{GGN16} or \cite[Proposition 2.2]{GI20} and integrating the equation \eqref{S-L} one has 
	\begin{equation}
	\label{AF2}
	\psi_{j,p_n}(\rho)=\rho^{\kappa_{j,p_n}}\int_{\rho}^1 s^{-1-2\kappa_{j,p_n}}\int_0^s t^{1+\kappa_{j,p_n}}V_{p_n}(t)\psi_{j,p_n}(t)\,dt ds
	\end{equation}
	where $\kappa_{j,p_n}=\sqrt{|\nu_j(p_n)|}>4$ by assumption. Observe that 
	\begin{align*}
	\left|\int_0^s t^{1+\kappa_{j,p_n}}V_{p_n}(t)\psi_{j,p_n}(t)\,dt\right|&\leq \| V_{p_n}\|_{\infty}\left|\int_0^s t^{-\tfrac12}\psi_{j,p_n}(t) t^{\kappa_{j,p_n}+\tfrac32}dt\right|\\		
	&\overset{\text{H\"older}}{\leq}\| V_{p_n}\|_{\infty}\left(\int_0^1\frac{(\psi_{j,p_n}(t))^2}{t}dt\right)^{\tfrac12}\left(\int_0^st^{3+2\kappa_{j,p_n}}dt\right)^{\tfrac12}\\
	&\overset{(\star)}{\leq}\varepsilon_{0,p}^{-2}\frac{s^{2+\kappa_{j,p_n}}}{\sqrt{4+2\kappa_{j,p_n}}}{\leq}\frac{\varepsilon_{0,p}^{-2}}{2}s^{2+\kappa_{j,p_n}},
	\end{align*}
where  $(\star) $ follows from  the normalization \eqref{norm1} and the fact  that $\| V_{p_n}\|_{\infty} \le \varepsilon_{0,p}^{-2}$ by  \eqref{V_p} and \eqref{epsilon}.
	Inserting this estimate in \eqref{AF2} we get
	\begin{equation}\label{AF3}
	|\psi_{j,p_n}(\rho)|\;\leq\; \frac{\varepsilon_{0,p}^{-2}}{2}\rho^{\kappa_{j,p_n}}\int_\rho^1 s^{1-\kappa_{j,p_n}}ds\;{\leq}\;\frac{\varepsilon_{0,p}^{-2}}{2}\rho^{\kappa_{j,p_n}}
	\frac{1-\rho^{2-\kappa_{j,p_n}}}{2-\kappa_{j,p_n}}\overset{ \kappa_{j,p_n}>4} {\leq} \varepsilon_{0,p}^{-2}\rho^2. 
	\end{equation}
	This implies that $\psi_{j,p_n}$ is continuous and differentiable in $\rho=0$ with $\psi_{j,p_n}(0)=(\psi_{j,p_n})'(0)=0$. Then we can integrate \eqref{S-L} in $(0,\rho)$ getting
	\[
	\rho(\psi_{j,p_n})'(\rho)=-\int_0^\rho\left(s V_{p_n}(s)+\frac{\nu_{j}(p_n)}{s}\right)\psi_{j,p_n}(s)\,ds.
	\]
	Combining with \eqref{AF3} we derive 
	\begin{eqnarray}\label{AF4}
	|(\psi_{j,p_n})'(\rho)|&\leq&\frac{\varepsilon_{0,p}^{-2}}{\rho}\int_0^\rho \left(s \|V_{p_n}\|_\infty+\frac{|\nu_{j}(p_n)|}{s}\right)s^2\,ds\nonumber\\
	&\overset{(*)}\leq& \frac{\varepsilon_{0,p}^{-2}}{\rho} \left(\varepsilon_{0,p}^{-2}\frac{\rho^4}{4}+C\frac{\rho^2}{2}\right)\nonumber\\
	&\leq& \varepsilon_{0,p}^{-2}\rho \left(\varepsilon_{0,p}^{-2}\rho^2+C\right),
	\end{eqnarray}
	where in $(*)$ we have used \eqref{V_p}, the fact that $\|V_{p_{n}}\|_{\infty}=p u_{p_{n}}(0)^{p-1}$ (since $\|u_{p}\|_{\infty}=u_{p}(0)$, cfr. \cite{{IS-arxiv}}), \eqref{epsilon} and 
	\eqref{A1}. 
	This implies that $\psi_{j,p_n}\in C^1[0,1]$. Furthermore by \eqref{S-L}
	\[
	-\psi_{j,p_n}''(\rho)=\frac{\psi_{j,p_n}'(\rho)}{\rho}+\left(\rho^2 V_{p_n}(\rho)+\nu_{j}(p_n)\right)\frac{\psi_{j,p_n}(\rho)}{\rho^2},\quad \mbox{ for }\rho\in(0,1],
	\]
	so using \eqref{AF3}, \eqref{AF4} and \eqref{A1}
	\begin{equation}
	\label{AF5}
	|\psi_{j,p_n}''(\rho)|\leq2\varepsilon_{0,p_n}^{-2}\left(\varepsilon_{0,p_n}^{-2}\rho^2+\widetilde C\right)\qquad\text{for }\rho\in(0,1].
	\end{equation}
	By \eqref{zeri-lim} for any $R>0$ there exists $n$ large enough such that  $R<\frac{r_{1,p_n}}{\varepsilon_{0,p_n}}$. Recalling the definition of the rescaled function \eqref{autof-resc}, by the regularity of $\psi_{j,p_n}$, we conclude that $\psi_{j,p_n}^0\in C^1[0,R]\cap C^\infty(0,R]$. Scaling into the estimates \eqref{AF3},  \eqref{AF4}, \eqref{AF5}  we obtain 
	that for $r\in[0,R]$:
	\[
	|\psi^0_{j,p_n}(r)|=|\psi_{j,p_n}(\e_{0,p_n} r)|\overset{\eqref{AF3}}{\leq}r^2,
	\]
	\[
	|(\psi^0_{j,p_n})'(r)|=\e_{0,p}|\psi_{j,p_n}'(\e_{0,p_n} r)|\overset{\eqref{AF4}}{\leq}(r^2+C)r\,{\leq}\,C_{R}r,
	\]	
	\[
	|(\psi^0_{j,p_n})''(r)|=\e_{0,p_n}^2|\psi_{j,p_n}''(\e_{0,p_n} r)|\overset{\eqref{AF5}}{\leq}2(r^2+\widetilde C)\,{\leq}\,\widetilde{C}_{R} \quad ,\text{ for }r\in(0,R]
	\]
thus $(\psi^0_{j,p_n})'$ are equicontinuous in $[0,R]$ and  Arzelà-Ascoli Theorem implies \eqref{lim-autof-risc-0}.		
\end{proof} 

The locally uniform convergence established in Lemma \ref{lemma:C1loc}  will be crucial to control the interactions among different scalings of the eigenfunction $\psi_{j,p}$. Adapting the proof of \cite[Lemma 3.7]{AG_N3}, we infer that 

\begin{lemma}\label{Lemma:NormaL2pesataPiccolaSuInsiemini} If $\nu_{j}(p)<-\frac{1}{2}$, then for any $\delta>0$ there exist $K(\delta)>1$ and $p(\delta,K)>1$ such that
	\[\int_{G_{p}(K)}\frac{(\psi_{j,p})^{2}}{r}dr \leq \delta, \qquad \mbox{ for } K\geq K(\delta) \mbox{ and }p\geq p(\delta,K) .\]
	Here $G_p(K)$ is the set defined in \eqref{def:G}.
\end{lemma}
\begin{proof} 
	By definition $G_{p}(K)=\bigcup_{i=0}^{m-1}[a_i,b_i]$ 
	where we set
	\begin{eqnarray*}
		&&a_i:=K\varepsilon_{i,p}\qquad i=0,\ldots, m-1\\
		&& b_i:=\left\{\begin{array}{lr}\frac{1}{K}\varepsilon_{i+1,p} &\quad i=0,\ldots, m-2\\
			1 &\quad i=m-1\end{array} \right.
	\end{eqnarray*}
	Thanks to Lemma \ref{lemma-stima-fp} and the definition \eqref{S-L-lim-autof}, one can chose $ K(\delta)$ and $p_1(\delta)$ such that
	\begin{equation}\label{scelta-delta}
	\max_{G_{p}(K)}f_{p} \le \frac{\delta}{4m} , \quad |K\eta^{i}(K)(\eta^{i})'(K)| \le \frac{\delta}{16m} , \quad |\frac1{K}\eta^{i}(\frac{1}K)(\eta^{i})'(\frac{1}K)| \le \frac{\delta}{16m} 
	\end{equation}
	for every $K>K(\delta)$, $p\ge  p_1(\delta)$, $i=0,\dots m-1$.
	
	Using $\psi_{j,p}$ as a test function in \eqref{S-L} and recalling the definition of $f_p$ in \eqref{def:f}	we get
	\begin{equation}\label{pezzi}
	\int_{a_i}^{b_i}\frac{(\psi_{j,p})^{2}}{r}dr=-\frac{1}{\nu_{j}(p)}\int_{a_i}^{b_i} (r \psi_{j,p}')'\psi_{j,p} dr - \frac{1}{\nu_{j}(p)}\int_{a_i}^{b_i}f_{p}(r)\frac{(\psi_{j,p})^{2}}{r}dr.\end{equation}
	Let us estimate the two integrals in the right hand side of \eqref{pezzi}. Concerning the first one
	\begin{equation}
	\label{secondoPezzetto} 
	- \frac{1}{\nu_{j}(p)}\int_{a_i}^{b_i}f_{p}(r)\frac{(\psi_{j,p})^{2}}{r}dr  \le 2\max_{G_{p}(K)}f_{p}\int_{0}^{1}\frac{(\psi_{j,p})^{2}}{r}dr\overset{\eqref{normalization}}{=}2\max_{G_{p}(K)}f_{p} \leq \frac{\delta}{2m}
	\end{equation}
	for every $K\ge K(\delta)$ and $p\ge p_1(\delta)$, thanks to \eqref{scelta-delta}. \\
	Moreover integrating by parts 
	\begin{align}
	- \frac{1}{\nu_{j}(p)}\int_{a_i}^{b_i} (r \psi_{j,p}')'\psi_{j,p} dr &=
	-\frac{1}{\nu_{j}(p)}\left[-\int_{a_i}^{b_i} r (\psi_{j,p}')^{2} dr+b_i\psi_{j,p}(b_i)\psi_{j,p}'(b_i) -a_i\psi_{j,p}(a_i)\psi_{j,p}'(a_i)\right] \nonumber\\
	& \leq 2|b_i\psi_{j,p}(b_i)\psi_{j,p}'(b_i)|+2|a_i\psi_{j,p}(a_i)\psi_{j,p}'(a_i)| \label{PrimoPezzetto}
	\end{align}
	since $\nu_j(p)<-\frac{1}{2}$.
	Observe that
	\[2  b_{m-1}\psi_{j,p}(b_{m-1})\psi_{j,p}'( b_{m-1})=2\psi_{j,p}(1)\psi_{j,p}'(1)=0.\]
	The other terms can be estimated by making use of Lemma \ref{lemma:C1loc}.
	For $i=0,\ldots, m-2$, rescaling according to $\varepsilon_{i,p}$ gives
	\[
	2|b_i\psi_{j,p}(b_i)\psi_{j,p}'(b_i)|=2\frac{1}{K}|\psi_{j,p}^{i+1}(\frac{1}{K})(\psi_{j,p}^{i+1})'(\frac{1}{K})|\overset{\eqref{lim-autof-risc-unif}}{\leq}2\frac{1}{K}|{\overline\psi^{i+1}_{j}(\frac{1}{K}) (\overline\psi^{i+1}_{j})'(\frac{1}{K})}|+\frac{\delta}{8m}
	\]
	after chosing $p\geq p_{2}(\delta, K)$, for a suitable $p_2(\delta,K)$.
	Similarly for $i=0,\ldots, m-1$ 
	\[2|a_i\psi_{j,p}(a_i)\psi_{j,p}'(a_i)| =2K|\psi_{j,p}^{i}(K)(\psi_{j,p}^{i})'(K)|
	\le	2K|{\overline\psi^{i}_{j}(K) (\overline\psi^{i}_{j})'(K)}| +\frac{\delta}{8m}
	\] for $p\geq p_{2}(\delta,K)$. 
	Summing up, \eqref{PrimoPezzetto} becomes
	\begin{align*}
	- \frac{1}{\nu_{j}(p)}\int_{a_{m-1}}^{b_{m-1}} (r \psi_{j,p}')'\psi_{j,p} dr & \le 		2K|{\overline\psi^{m-1}_{j}(K) (\overline\psi^{m-1}_{j})'(K)}| 
	+\frac{\delta}{8m} ,
	\\
	- \frac{1}{\nu_{j}(p)}\int_{a_i}^{b_i} (r \psi_{j,p}')'\psi_{j,p} dr & 	\le 	2\frac{1}{K}|{\overline\psi^{i+1}_{j}(\frac{1}{K}) (\overline\psi^{i+1}_{j})'(\frac{1}{K})}|
	+ 	2K|{\overline\psi^{i}_{j}(K) (\overline\psi^{i}_{j})'(K)}| 
	+\frac{\delta}{4m} 
	\end{align*}
	if $i=0,\dots m-2$.
	We remark that, according to Corollary \ref{remark}-$(ii)$, at most one between the limit functions $\overline\psi^{i}_{j}$ and $\overline\psi^{i+1}_{j}$ differs from zero, and either $\overline\psi^{i}_{j}=A_j \eta^i$ or $\overline\psi^{i+1}_j=A_j \eta^{i+1}$, with $|A_j|\le 1$.  
	Therefore \eqref{scelta-delta} implies that for every $i=0,\dots m-1$
	\begin{equation}
	\label{altroPezzetto}
	- \frac{1}{\nu_{j}(p)}\int_{a_i}^{b_i} (r \psi_{j,p}')'\psi_{j,p} dr \leq \frac{\delta}{2m}\end{equation}
	for $K\ge K(\delta)$  and for every $p\geq p_{2}(\delta,K)$.
	\\
	Substituting the estimates \eqref{secondoPezzetto} and \eqref{altroPezzetto} into \eqref{pezzi} we deduce that
	\[\int_{a_i}^{b_i}\frac{(\psi_{j,p})^{2}}{r}dr\leq \frac{\delta}{m},\]
	for $K\ge K(\delta)$  and for every $p\geq \max\{{p_1(\delta)}\, , \, p_{2}(\delta, K)\}$. The conclusion follows summing up for $i=0,\ldots, m-1$.
\end{proof}

\

\

\subsection {Proof of Proposition \ref{prop:BaseInduttiva}}\label{sottoSectionproof}

Proposition \ref{prop:BaseInduttiva} follows by adapting the arguments in \cite[Proposition 3.4]{AG-N2}, which concernes the case of two nodal zones. For the reader's comprehension we report a detailed proof. First we obtain an estimate from above of $\nu_1(p)$ in Lemma \ref{lemma:stimaPrimoAutovalore}. Next we conclude the proof relying on the general convergence result in Proposition \ref{lem:conv-autof-resc} and in particular on Corollary \ref{remark} and Lemma \ref{lemma:C1loc}.

\begin{lemma}
	\label{lemma:stimaPrimoAutovalore} 
	\[\limsup\limits_{p\to \infty} \nu_1(p) \le \beta^{m-1}.\]
\end{lemma}

\begin{proof} 
	From the variational characterization \eqref{var-char-1}, it suffices to exhibit	for every $0<\e<1$ a sequence $\varphi_p\in \mathcal H_{0,\mathrm{rad}}$ such that 
	\begin{equation}\label{goal}
	\nu_1(p) \le \dfrac{\int_0^1 r\left(|\varphi_p'|^2 -V_p\varphi_p^2 \right) dr}{\int_0^1  r^{-1}\varphi_p^2 dr } \le  \beta^{m-1} + \e 
	\end{equation}
	if $p$ is large enough.
	So we pick a cut-off function $\Phi\in C^{\infty}_0(0,\infty)$ such that  
	\begin{equation}\label{eq:cut-off}
	\begin{array}{ll}
	0\le \Phi(r) \le 1 , \ &  \Phi(r)=\begin{cases}
	1 &  \text{ if } \frac{1}{R} <   r< R , \\
	0 &\text{ if  } 0\le r < \frac{1}{2R} \text{ or }  r> {2R} ,
	\end{cases} \\
	&  \left| \Phi'(r)\right|\leq  \begin{cases}  2R  & \text{ if } \frac{1}{2R} < r< \frac{1}{R} ,\\ \frac 2{R}  & \text{ if } R< r < 2R.
	\end{cases}
	\end{array}\end{equation}
	Letting $\e_p=\e_{m-1,p}$ and $\eta=\eta^{m-1}$ as defined in \eqref{epsilon} and  \eqref{S-L-lim-autof}, respectively, we set 	
	\begin{align}
	\label{psi-epsilon}
	\varphi_p(r)=\eta\left(\frac{r}{\e_p}\right) \Phi \left(\frac r{ \e_p} \right)  , \quad \text{ as } r\in[0,1] .
	\end{align}
	The function $\eta$  is increasing and  decreasing on an interval $(0, a)$ and $\left(a, \infty\right)$ respectively, moreover   $\lim\limits_{s\to 0}\eta(s)=0$, $\lim\limits_{s\to \infty}\eta(s)=0$, and $\int\limits_0^{\infty} s^{-1} \eta^2 ds =1$. So we can choose $R=R(\e)$ in such a way that
	\begin{align}\label{R-e-1}
	\eta(s)\le \eta(\frac{1}{R})<\frac{\e}{4} \quad \text{ for  $s< \frac{1}{R}\quad$ and }\quad \eta(s)\le \eta(R)<\frac{\e}{4} \quad \text{ for  $s>R$, }\\
	\label{R-e-2}
	\int_0^{\infty} s^{-1} \eta^2 \Phi^2 ds \ge\int_{\frac{1}{R}}^{R} s^{-1} \eta^2  ds  \ge  1-\e /8. 
	\end{align}
	Notice that since $\e_p\to 0$ we may assume w.l.g.~that $p$ is so large that $1/\e_p > 2R$, so that $\varphi_p\in \mathcal H_{0,\mathrm{rad}}$.
	
	Inserting the test function $\varphi_p$ in the variational characterization \eqref{var-char-1} of $\nu_1(p)$ we have
	\begin{equation}\label{IF1.0}
	\nu_1(p) \le \dfrac{\int_0^1 r\left(|\varphi_p'|^2 -V_p\varphi_p^2 \right) dr}{\int_0^1  r^{-1}\varphi_p^2 dr, }
	\end{equation}
	Next we estimate all the terms.\\
	Using the relation $[(fg)']^2= f'(fg^2)'+f^2(g')^2$, scaling with respect to $\e$ and using the equation \eqref{S-L-lim} satisfied by $\eta$ (recall that $\Phi$ has compact support) one gets
	\begin{eqnarray}\label{IF1.1}
	\int_0^1 r|\varphi_p'|^2\,dr &=&\int_0^1 r \left[\left(\eta\Big(\frac{r}{\e_p}\Big) \Phi \Big(\frac{r}{\e_p}\Big) \right)'\right]^2dr\nonumber
	\\
	&=&
	\frac{1}{\e_p}\int_0^1 r   \eta'\Big(\frac{r}{\e_p}\Big)  
	\left( \eta\Big(\frac{r}{\e_p}\Big) \Phi^2 \Big(\frac{r}{\e_p}\Big)\right)'dr\  +  \frac{1}{\e_p^2} \ \int_0^1 r \eta^2\Big(\frac{r}{\e_p}\Big) \left(\Phi'\Big(\frac{r}{\e_p}\Big)\right)^2dr\nonumber
	\\
	&=&\int_0^{
		\frac{1}{\e_p}} s  \eta'  
	\left( \eta \Phi^2 \right)'ds\  +  \ \int_0^{
		\frac{1}{\e_p}} s \eta^2 \left(\Phi'\right)^2ds\nonumber
	\\
	&\overset{\eqref{S-L-lim}}{=} & \beta^{m-1} \int_0^{\infty}s^{-1}\eta^2\Phi^2\, ds +\int_0^{\infty}se^{Z_{m-1}}\eta^2\Phi^2\, ds +  \ \int_0^{
		\infty} s \eta^2 \left(\Phi'\right)^2ds
	\end{eqnarray}
	and by the choice of $\Phi$ we have 
	\begin{align} \nonumber
	\int_0^{\infty} s \eta^2 (\Phi')^2 ds &\leq {4}{R^2}\int_{\frac{1}{2R}}^{\frac{1}{R}}  s \eta^2 ds +
	\frac{4}{R^2}\int_R^{2R}  s \eta^2 ds  \\ 
	&\overset{\eqref{R-e-1}}{<}  \frac{\e^2R^2}{4}\int_{\frac{1}{2R}}^{\frac{1}{R}}  s \, ds +\frac{\e^2 }{4 R^2}  \int_R^{2R} s\,  ds = \frac{3\e^2}{4}  < \frac{3\e}{4}.\label{IF1.2}
	\end{align}
	Furthermore scaling with respect to $\varepsilon_p$, since $\frac{1}{\varepsilon_p}>2R$ we get
	\begin{equation}\label{IF1.3}
	\int_0^1V_p\varphi_p^2 dr=\int_0^{\infty}s V_p^{m-1}\eta^2 \Phi^2 ds
	\end{equation}
	and
	\begin{equation}\label{IF1.3bis}
	\int_0^1  r^{-1}\varphi_p^2 dr=\int_0^{\infty}  s^{-1}\eta^2 \Phi^2  ds.
	\end{equation}
	%
	Inserting \eqref{IF1.1}, \eqref{IF1.2}, \eqref{IF1.3} and \eqref{IF1.3bis} in \eqref{IF1.0} we obtain
	\begin{align*}
	\nu_1(p) & < \beta^{m-1} +  \dfrac{\int_0^{\infty} s (V^{m\!-\!1}_p-  e^{Z_{m-1}}) \eta^2  \Phi^2  ds + \frac{3\e}{4}  }
	{\int_0^{\infty}  s^{-1}\eta^2 \Phi^2  ds} \\
	& \overset{\eqref{R-e-2}}{<} \beta^{m-1} + \frac{\int_0^{+\infty} s  \left|e^{Z_{m-1}}- V^{m\!-\!1}_p\right| \eta^2\Phi^2 ds+ \frac{3}{4} \e}{1-\e /8} .
	\end{align*}
	On the other hand by the properties of $\Phi$  we have 
	\begin{align*}
	\int_0^{+\infty} s  \left| V^{m\!-\!1}_p - e^{Z_{m-1}}\right| \eta^2\Phi^2 ds \le \sup_{(\frac{1}{2R},2R)}| V^{m\!-\!1}_p- e^{Z_{m-1}}|\int_{\frac{1}{2R}}^{2R}s \eta^2 ds 
	\end{align*}
	and since $ V^{m\!-\!1}_p\to e^{Z_{m-1}}$ uniformly on $[\frac{1}{2R},2R]$ we can take $p_{\e}$ in dependence by $\e$ and $R(\e)$ large enough such that 
	\[ \sup_{(\frac{1}{2R},2R)}| V^{m\!-\!1}_p- e^{Z_{m-1}}| \le \frac{\e }{8\int_{\frac{1}{2R}}^{2R}s \eta^2 ds }  \quad \text{ for } p> p_{\e},\]
	which concludes the proof of \eqref{goal}.
	
\end{proof}

\

\

\begin{proof}[Proof of Proposition \ref{prop:BaseInduttiva}] By Lemma \ref{lemma:stimaPrimoAutovalore} and \eqref{ordine-beta}
	\[
	\limsup_{p\rightarrow +\infty} \nu_{1}(p)\leq\beta^{m-1} <\beta^i\qquad i=0,\ldots,m-2.
	\] 
	As a consequence,  Corollary \ref{remark}-{\it
		(i)} implies that  
	\begin{equation}\label{tutteNulle}
	\overline\psi^i_1 \equiv 0\qquad \mbox{ as }i= 0, \dots m-2.\end{equation}
	So, by Corollary \ref{remark}-{\it (ii)},  Proposition \ref{prop:BaseInduttiva} is proved after checking that  
	\begin{equation}
	\label{eta1NoZero}
	\overline\psi^{m-1}_1 \not\equiv 0.
	\end{equation} 

	\noindent
	To this aim we fix $\delta >0$ such that $\delta<-\beta_{m-1}/3$ and $K=K(\delta)$ and $p (\delta)$ as in Lemma \ref{lemma-stima-fp}. By the definition of $\nu_1(p)$ it follows that 
	\begin{eqnarray*}
		-\nu_1(p)&=&-\int_0^1 r[(\psi_{1,p}')^2-V_p(\psi_{1,p})^2]dr\\
		&\leq &\int_0^1 rV_p(\psi_{1,p})^2dr\\
		&=& \int_{G_p(K)}  rV_p(\psi_{1,p})^2 dr+\int_{0}^{K\varepsilon_{0,p}}
		rV_p(\psi_{1,p})^2 dr  +\sum_{i=1}^{m-1} \int_{\frac{1}{K}\varepsilon_{i,p}}^{K\varepsilon_{i,p}} rV_p(\psi_{1,p})^2 dr 
		\\
		&=&   I_1(p)+I_2(p)+I_3(p)
	\end{eqnarray*}
	The normalization of the eigenfunction and the estimate obtained in Lemma \ref{lemma-stima-fp} assure that
	\begin{align*}
	I_1(p)=& \int_{G_p(K)}  rV_p(\psi_{1,p})^2 dr =\int_{G_p(K)}  f_{p} \frac{(\psi_{1,p})^2}{r}dr 
	\\ \leq &\sup_{G_p(K)}f_p \int_0^1 \frac{(\psi_{1,p})^2}{r}
	dr
	= \sup_{G_p(K)}f_p \leq \delta
	\end{align*}
	for $p\geq p(\delta)$.
	\\
	Observe that, by Lemma \ref{lemma:C1loc}, $ \psi^i_{1,p}\to \overline\psi_{1}^{i}$  in $C^1_{\loc}(0,+\infty)$ for $i=1,\dots m-1$, and in $C^1_{\loc}[0,+\infty)$ for $i=0$. Indeed $\bar\nu_1:=\limsup_{p\rightarrow +\infty}\nu_1(p)\leq\beta^{m-1}\leq\beta^{1}<-25$ by \eqref{ordine-beta}.
Furthermore, by \eqref{pot-lim-i} and \eqref{pot-lim-0}, 
	$ V^i_p \to e^{Z_i}$ in $C^0_{\loc}(0,+\infty)$ for $i=1,\ldots,m-1$, or respectively in $C^0_{\loc}[0,+\infty)$ for $i=0$.
	Hence, rescaling the second integral according to $\e_{0,p}$ gives
	\begin{align*}
	I_2(p)=& \int_{0}^{K\varepsilon_{0,p}}
	rV_p(\psi_{1,p})^2 dr = \int_0^K r  V^0_p(\psi^0_{1,p})^2 dr
	=
	\int_0^K r  e^{Z_0} (\overline\psi_{1}^{0})^2 dr +o_{p}(1) \overset{\eqref{tutteNulle}}{\leq}\delta,
	\end{align*}
	if $p\geq p_{2}(\delta)$. 
	Similarly, for what concerns the  third term,
	\begin{align*}
	I_3(p)=& \sum_{i=1}^{m-1} \int_{\frac{1}{K}\varepsilon_{i,p}}^{K\varepsilon_{i,p}} rV_p(\psi_{1,p})^2 dr =
	\sum\limits_{i=1}^{m-1} \int_{\frac 1K}^K r  V^i_p (\psi^i_{1,p})^2 dr
	\\
	=&\sum_{i=1}^{m-1} \int_{\frac{1}{K}}^{K} r e^{Z_i}(\overline\psi_{1}^i)^2 dr + o_{p}(1)\\
	\overset{\eqref{tutteNulle}}{\leq}& \int_{\frac{1}{K}}^{K} r e^{Z_{m-1}}(\overline\psi_{1}^{m-1})^2 dr + \delta,
	\end{align*}
	for $p\geq p_{3}(\delta)$.
	Summing up, taking $\bar p=\max\{p(\d), p_2(\d),p_3(\d)\}$ we have
	\[\int_{\frac 1K}^K re^{Z_{m-1}}(\overline\psi^{m-1}_{1})^2 dr\geq -\nu_1(p)-3\d \ \text{  for }p>\bar p\]
	so, passing to the $\liminf$  and using Lemma \ref{lemma:stimaPrimoAutovalore},
	\[\int_{\frac 1K}^K re^{Z_{m-1}}(\overline\psi^{m-1}_{1})^2 dr\geq-\limsup_{{p\rightarrow+\infty}} \nu_1(p)-3\d \ge  -\beta^{m-1}-3\d >0  
	\]
	by the choice of $\delta$. Hence $\overline\psi^{m-1}_1\neq 0$, concluding the proof. 
\end{proof}

\

\

\subsection{Proof of Proposition \ref{prop:autovAutofIteration}}\label{section:AUTOVALORI_{}INTERMEDI}
Computing the limits of the subsequent eigenvalues is more involved, and it is done in an iterative way.  Similarly as in Section \ref{sottoSectionproof}, also here we follow a two step scheme: first we obtain an estimate from above by producing a suitable test function (Lemma \ref{lemma:stimaAutovalorehesimo}), then we conclude the proof of Proposition \ref{prop:autovAutofIteration} by exploiting the convergence results in Proposition \ref{lem:conv-autof-resc} and taking advantage of   the orthogonality condition \eqref{normalization}.

\begin{lemma} 
	\label{lemma:stimaAutovalorehesimo} Let $h\in\{2,\ldots, m\}$ and assume that Theorem \ref{thm:autovaloriAutofunzioni} holds true for any $j=1, \ldots, h-1$. Then
	\begin{equation}\label{est-lim-rse}\limsup\limits_{p\to \infty} \nu_h(p) \le \beta^{m-h}.\end{equation}
\end{lemma}

\begin{proof}
	By the variational characterization \eqref{var-char}, it suffices to exhibit	for every $0<\e<1$ a function $\varphi_p\in \mathcal H_{0,\mathrm{rad}}$, $\varphi_p\underline{\perp} \psi_{1,p}, \psi_{2,p},\ldots, \psi_{h-1,p}$ such that 
	\begin{equation}\label{goalh}
	\dfrac{\int_0^1 r\left(|\varphi_p'|^2 -V_p\varphi_p^2 \right) dr}{\int_0^1  r^{-1}\varphi_p^2 dr } \le  \beta^{m-h} + \e 
	\end{equation}
	if $p$ is large enough.
	Let $\Phi=\Phi_R$ be the  cut-off function defined in \eqref{eq:cut-off}, $\e_p=\e_{m-h,p}$ and $\eta=\eta^{m-h}$ as defined in \eqref{epsilon} and  \eqref{S-L-lim-autof}, respectively, and set 	
	\begin{align}
	\label{psi-epsilon}
	\varphi_p(r)=\eta\Big(\frac{r}{\e_p}\Big) \Phi \Big(\frac r{ \e_p} \Big)+\sum_{j=1}^{h-1} a_{j,p}\psi_{j,p}  , \quad \text{ as } r\in[0,1],
	\end{align}
	with $R=R(\e)$ satisfying \eqref{R-e-1}, \eqref{R-e-2} and $a_{j,p}\in\mathbb R$  choosen so that $\varphi_p\underline{\perp} \psi_{1,p}, \psi_{2,p},\ldots, \psi_{h-1,p}$, namely:
	\begin{equation}\label{def_a_pi}
	a_{j,p}:=-\int_0^1 r^{-1}\psi_{j,p}(r)\eta\Big(\frac{r}{\e_p}\Big) \Phi \Big(\frac r{ \e_p} \Big)dr.\end{equation}

	Notice that since $\e_p\to 0$ we may assume w.l.g.~that $p$ is so large that $1/\e_p > 2R$, so that $\varphi_p\in \mathcal H_{0,\mathrm{rad}}$.
	
	Furthermore 
	\begin{equation} \label{aiptoZero} a_{j,p}\to 0 \mbox{ as }p\to \infty.\end{equation}
	Indeed,  rescaling  w.r.t. $\e_p$,  using that $\Phi$ has compact support and that the interval $\left(\frac{r_{m-h,p}}{\e_p}, \frac{r_{m-h+1,p}}{\e_p}\right)$   invades $(0,\infty)$ by \eqref{zeri-lim}, we can write for $p$ large
	\begin{eqnarray*}
		\int_0^1 r^{-1}\psi_{j,p}(r)\eta\Big(\frac{r}{\e_p}\Big) \Phi \Big(\frac r{ \e_p} \Big)\, dr  
		&=& \int_0^{r_{m-h,p}} r^{-1}\psi_{j,p}(r)\eta\Big(\frac{r}{\e_p}\Big) \Phi \Big(\frac r{ \e_p} \Big)\, dr\\
		&&+  \int_{r_{m-h,p}}^{r_{m-h+1,p}} r^{-1}\psi_{j,p}(r)\eta\Big(\frac{r}{\e_p}\Big) \Phi \Big(\frac r{ \e_p} \Big)\, dr\\
		&& + \int_{r_{m-h+1,p}}^1 r^{-1}\psi_{j,p}(r)\eta\Big(\frac{r}{\e_p}\Big) \Phi \Big(\frac r{ \e_p} \Big)\, dr\\
		&=&\int_{r_{m-h,p}/\e_{p}}^{r_{m-h+1,p}/\e_{p}} s^{-1} \, \psi_{j,p}^{m-h}\eta \, \Phi   \, ds  .
	\end{eqnarray*} 
	\eqref{aiptoZero} then follows passing to the limit
	as $p\to \infty$ and using that $\psi_{j,p}^{m-h}\to \overline\psi_j^{m-h}$ weakly in 
	$\mathcal D_{\mathrm{rad}}$ by Proposition \ref{lem:conv-autof-resc} and that $\overline\psi_j^{m-h}=0$  for $j=1,\ldots, h-1$ by assumption.
	
	\
	
	We want to estimate all the integrals in the left hand side of \eqref{goalh}. Observe that
	\begin{eqnarray}\label{IF1}
	\int_0^1 r|\varphi_p'|^2dr
	&=&\int_0^1 r \left[\left(\eta\Big(\frac{r}{\e_p}\Big) \Phi \Big(\frac{r}{\e_p}\Big)\right)'\right]^2dr
	+\sum_{j=1}^{h-1}a_{j,p}^2\int_0^1 r (\psi_{j,p}')^2dr\nonumber\\
	&&
	+ 2\sum_{j=1}^{h-1}a_{j,p}\int_0^1  r \left(\eta\Big(\frac{r}{\e_p}\Big) \Phi \Big(\frac{r}{\e_p}\Big)\right)' \psi_{j,p}' dr \nonumber\\
	&& + \sum_{j,\ell=1, \ j\neq \ell}^{h-1}a_{j,p}a_{\ell,p}\int_0^1  r 
	\psi_{j,p}' \psi_{\ell,p}' dr\nonumber
	\\
	&=&A_p+ B_p +C_p+D_p
	\end{eqnarray} 
	and similarly that
	\begin{eqnarray}\label{IF2}
	\int_0^1 rV_p\varphi_p^2  dr
	&=&\int_0^1 rV_p(r) \eta\Big(\frac{r}{\e_p}\Big)^2 \Phi \Big(\frac{r}{\e_p}\Big)^2dr
	+\sum_{j=1}^{h-1}a_{j,p}^2\int_0^1 r V_p(r)(\psi_{j,p})^2dr\nonumber\\
	&&
	+ 2\sum_{j=1}^{h-1}a_{j,p}\int_0^1  rV_p(r) \eta\Big(\frac{r}{\e_p}\Big) \Phi \Big(\frac{r}{\e_p}\Big) \psi_{j,p} dr\\
	&&
	+\sum_{j,\ell=1, \ j\neq \ell}^{h-1}a_{j,p}a_{\ell,p}\int_0^1  r V_p(r)
	\psi_{j,p} \psi_{\ell,p} dr\nonumber
	\\
	&=&E_p+ F_p +G_p+H_p.
	\end{eqnarray}
	The same computations as in Lemma \ref{lemma:stimaPrimoAutovalore} (see \eqref{IF1.1} and \eqref{IF1.3}) show that		
	\begin{eqnarray}\label{IF3e4}
	A_p-E_p&=&\beta^{m-h}\int_0^{\infty}s^{-1}\eta^2\Phi^2\, ds +\int_0^{\infty}se^{Z_{m-h}}\eta^2\Phi^2\, ds +  \ \int_0^{
		\infty} s \eta^2 \left(\Phi'\right)^2ds\nonumber\\
	&&-\int_0^{\infty}s V_p^{m-h}\eta^2 \Phi^2 ds.
	\end{eqnarray}
	Next using that $\psi_{j,p}$ solves \eqref{S-L} and \eqref{normalization}, and recalling the definition of $a_{j,p}$ in \eqref{def_a_pi}, we have
	\begin{eqnarray}
	B_p-F_p&=&\sum_{j=1}^{h-1}a_{j,p}^2\nu_j(p);\label{IF5}\\
	C_p-G_p&=&2\sum_{j=1}^{h-1}a_{j,p}\nu_j(p)\int_0^1r^{-1}\eta\Big(\frac{r}{\e_p}\Big) \Phi \Big(\frac{r}{\e_p}\Big) \psi_{j,p} dr\nonumber\\
	& \overset{\eqref{def_a_pi}}{=}&-2 \sum_{j=1}^{h-1}a_{j,p}^2\nu_j(p);\label{IF6}\\
	D_p -H_p&=&\sum_{j,\ell=1, j \neq \ell}^{h-1}a_{j,p}a_{\ell,p}\nu_{\ell}(p)\int_0^1r^{-1}\psi_{j,p} \psi_{\ell,p} dr=0\label{IF7}.
	\end{eqnarray}
	Hence substituting \eqref{IF3e4}, \eqref{IF5}, \eqref{IF6} and \eqref{IF7} in \eqref{IF1} and \eqref{IF2} we infer:
	\begin{eqnarray}\nonumber
	\int_0^1 r|\varphi_p'|^2dr-\int_0^1 rV_p\varphi_p^2  dr &=&
	\beta^{m-h}\int_0^{\infty}s^{-1}\eta^2\Phi^2\, ds  +  \int_0^{
		\infty} s \eta^2 \left(\Phi'\right)^2ds\\
	&& 
	+\int_0^{\infty}s(e^{Z_{m-h}}- V_p^{m-h})\eta^2\Phi^2\, ds\nonumber\\
	&& -\sum_{j=1}^{h-1}a_{j,p}^2\nu_j(p).
	\label{num}
	\end{eqnarray}
	On the other hand using once more \eqref{normalization} and \eqref{def_a_pi}, rescaling with respect to $\e_p$ and using the properties of $\Phi$ it also  follows that 
	\begin{eqnarray} \nonumber
	\int_0^1r^{-1}\varphi_p^2 dr &=&\int_0^1 r^{-1}\left(\eta\left(\frac r{ \e_p} \right)\Phi\left(\frac r{ \e_p} \right)\right)^2 dr  +  \sum\limits_{j,  \ell=1}^{h-1}a_{j,p} a_{\ell,p} \int_0^1 r^{-1}\psi_{j,p} \psi_{\ell,p} dr \\ \nonumber 
	&&+2 \sum\limits_{j=1}^{h-1}a_{j,p}\int_0^1 r^{-1} \psi_{j,p}\eta\left(\frac r{ \e_p}\right)\Phi\left(\frac r{\e_p}\right) dr \\\nonumber 
	&=& \int_0^1r^{-1}\left(\eta\left(\frac r{ \e_p} \right)\Phi\left(\frac r{ \e_p} \right)\right)^2 dr- \sum\limits_{j=1}^{h-1}a_{j,p}^2\\ 
	\label{den}
	&=& \int_0^{\infty} s^{-1}\eta^2\Phi^2  ds  - \sum\limits_{j=1}^{h-1}a_{j,p}^2.
	\end{eqnarray}
	Inserting \eqref{num} and \eqref{den}  into the l.h.s. of \eqref{goalh} we get
	\begin{eqnarray*}
		&&\frac{\int_0^1 r\left(|\varphi_p'|^2 -V_p\varphi_p^2 dr\right) dr}{\int_0^1  r^{-1}\varphi_p^2 dr }	
		=\\
		&&	\;\;\quad = \quad \beta^{m-h} + \frac{\int_0^{\infty} s (V^{m-h}_p-  e^{Z_{m-h}}) \eta^2  \Phi^2  ds + \int_0^{\infty} s \eta^2 (\Phi')^2 ds- \sum\limits_{j=1}^{h-1}a_{j,p}^2(\nu_j(p)-\beta^{m-h})}{\int_0^{\infty}  s^{-1}\eta^2 \Phi^2 ds- \sum\limits_{j=1}^{h-1}a_{j,p}^2 }
		\\
		&&\overset{\eqref{aiptoZero}+\eqref{A1}}{=}\beta^{m-h}+\frac{\int_0^{\infty} s (V^{m-h}_p-   e^{Z_{m-h}}) \eta^2  \Phi^2  ds + \int_0^{\infty} s \eta^2 (\Phi')^2 ds+o(1)}{\int_0^{\infty}  s^{-1}\eta^2 \Phi^2 ds +o(1) }
		\\
		&&\overset{\eqref{R-e-2}+\eqref{IF1.2}}{\leq}\beta^{m-h}+\frac{\int_0^{\infty} s (V^{m-h}_p-  e^{Z_{m-h}}) \eta^2  \Phi^2  ds + \tfrac{3}{4}\e +o(1)}{1 -\tfrac{\e}{8} +o(1) }.
	\end{eqnarray*}
	On the other hand, similarly as at the end of the proof of Lemma \ref{lemma:stimaPrimoAutovalore}, one can prove that 
	\begin{align*}
	\int_0^{+\infty} s  \left| V^{m\!-\!h}_p - e^{Z_{m-h}}\right| \eta^2\Phi^2 ds \le \sup_{(\frac{1}{2R},2R)}| V^{m\!-\!h}_p- e^{Z_{m-h}}|\int_{\frac{1}{2R}}^{2R}s \eta^2 ds \leq \frac{\varepsilon}{8},
	\end{align*}
	which concludes the proof of \eqref{goalh}.
\end{proof}

Next we conclude the proof of Proposition \ref{prop:autovAutofIteration} by exploiting the orthogonality  condition \eqref{normalization}, which allows to pick up, among all the rescaled functions introduced in \eqref{autof-resc}, the only one which has a nontrivial limit. 

\begin{proof}[Proof of Proposition \ref{prop:autovAutofIteration}]
	Fix  $h\in\{2,\ldots , m-1\}$, we want to prove that:
		\begin{equation}\label{limh}
		\lim_{p\rightarrow +\infty}\nu_{h}(p)=\beta^{m-h}
		\end{equation}
		and that there exists $A_{h}\neq 0$ such that
		\begin{eqnarray*}
			&& \overline\psi_h^{m-h}=A_{h}\eta^{m-h};\\
			&&\overline\psi_h^{i}= 0,\qquad i=0,\ldots, m-1, \ i\neq m-h.
		\end{eqnarray*}
	
	 By Lemma \ref{lemma:stimaAutovalorehesimo} and \eqref{ordine-beta}
		$\limsup_{p\rightarrow +\infty} \nu_{h}(p)\leq\beta^{m-h}<\beta^{m-i}$ for $i=h+1, \dots m $, then Corollary \ref{remark}-{\it(i)} implies that  
		\begin{equation}\label{primiEtaNulli}
		\overline\psi_{h}^{m-i}= 0, \qquad i=h+1, \dots m.
		\end{equation}
		Furthermore the claim follows by showing that $\overline\psi_h^{m-h}\neq 0$, thanks to Corollary \ref{remark}-{\it(ii)}.
		So we assume by contradiction that 
		\begin{equation}\label{altroEnaNullo}
		\overline\psi_{h}^{m-h}= 0.
		\end{equation}
		As a preliminary step, we will deduce from \eqref{primiEtaNulli}, \eqref{altroEnaNullo} that there exists $\kappa  \in\{ 1,\dots h-1 \}$  such that
		\begin{equation}
		\label{etam-1noZero}
		\overline\psi_{h}^{m-\kappa}\neq 0 .
		\end{equation}

	In order to prove \eqref{etam-1noZero} let us fix $\delta >0$ such that $\delta<-\beta^{m-h}/3$ and $K=K(\delta)$ as in Lemma \ref{lemma-stima-fp}. By the definition of of $\nu_{h}(p)$ it follows that
	\begin{eqnarray*}
		-\nu_h(p)&=&-\int_0^1 r[(\psi_{h,p}')^2-V_p(\psi_{h,p})^2]dr\\
		&\leq &\int_0^1 rV_p(\psi_{h,p})^2dr\\
		&=& \int_{G_p(K)}  rV_p(\psi_{h,p})^2 dr+\int_{0}^{K\varepsilon_{0,p}}
		rV_p(\psi_{h,p})^2 dr  +\sum_{i=1}^{m-1} \int_{\frac{1}{K}\varepsilon_{i,p}}^{K\varepsilon_{i,p}} rV_p(\psi_{h,p})^2 dr \\
		&=&   I_1(p)+I_2(p)+I_3(p) .
	\end{eqnarray*}
	
	We estimate these three terms with arguments similar  to the ones exploited in  the proof of Proposition \ref{prop:BaseInduttiva}. Indeed the normalization of the eigenfunction and the estimate obtained in Lemma \ref{lemma-stima-fp} assure that
	\begin{align*}
	I_1(p)=& \int_{G_p(K)}  rV_p(\psi_{h,p})^2 dr =\int_{G_p(K)}  f_{p} \frac{(\psi_{h,p})^2}{r}dr 
	\\ \leq &\sup_{G_p(K)}f_p \int_0^1 \frac{(\psi_{h,p})^2}{r}
	dr
	= \sup_{G_p(K)}f_p \leq \delta
	\end{align*}
	for $p\geq p_1(\delta)$.
	\\
	Moreover, by Lemma \ref{lemma:C1loc}, 
\begin{eqnarray}
\label{convergenzebuoneh}
 \psi^i_{h,p}\to \overline\psi_{h}^{i}  &&\mbox{ in }C^1_{\loc}(0,+\infty) \mbox{ for }i=1,\dots m-1, \\
\label{convergenzebuoneh0}&&\mbox{  in }C^1_{\loc}[0,+\infty)\mbox{ for }i=0.
\end{eqnarray}
	Indeed $\bar\nu_h:=\limsup_{p\rightarrow +\infty}\nu_h(p)\leq\beta^{m-h}\leq \beta^{1}<-25$ by \eqref{ordine-beta}.\\ Furthermore, by \eqref{pot-lim-i} and \eqref{pot-lim-0}, 
	$ V^i_p \to e^{Z_i}$ in $C^0_{\loc}(0,+\infty)$ for $i=1,\ldots,m-1$ , in $C^0_{\loc}[0,+\infty)$ for $i=0$.\\
	Hence, rescaling the second integral according to $\e_{0,p}$ gives
	\begin{align*}
	I_2(p)=& \int_{0}^{K\varepsilon_{0,p}}
	rV_p(\psi_{h,p})^2 dr = \int_0^K r  V^0_p(\psi^0_{h,p})^2 dr
	\overset{\eqref{lim-autof-risc-0}}{=}
	\int_0^K r  e^{Z_0} (\overline\psi_{h}^{0})^2 dr +o_{p}(1) \overset{\eqref{primiEtaNulli}}{\leq}\delta,
	\end{align*}
	if $p\geq p_{2}(\delta, K)$. 
	Similarly, for what concerns the  third term, 
	\begin{eqnarray*}
		I_3(p)&=& \sum_{i=1}^{m-1} \int_{\frac{1}{K}\varepsilon_{i,p}}^{K\varepsilon_{i,p}} rV_p(\psi_{h,p})^2 dr =
		\sum\limits_{i=1}^{m-1} \int_{\frac 1K}^K r  V^i_p (\psi^i_{h,p})^2 dr\\
		&\overset{\eqref{lim-autof-risc}}{=}&\sum_{i=1}^{m-1} \int_{\frac{1}{K}}^{K} r e^{Z_i}(\overline\psi_{h}^i)^2 dr + o_{p}(1)\\
		&\overset{\eqref{primiEtaNulli},  \eqref{altroEnaNullo}}{\leq} &   \sum_{\kappa=1}^{h-1} \int_{\frac{1}{K}}^{K} r e^{Z_{m-\kappa}} (\overline\psi_{h}^{m-\kappa})^2 dr +\d,
	\end{eqnarray*} 
	for $p\geq p_3(\delta,K)$.
	Summing up we then get
	\[\sum_{\kappa=1}^{h-1} \int_{\frac{1}{K}}^{K} r e^{Z_{m-\kappa}}(\overline\psi_{h}^{m-\kappa})^2 dr \geq -\nu_h(p)-3\delta, \qquad \mbox{for }p\geq\bar p:=\max\{p_{1}(\delta),p_{2}(\delta,K),p_{3}(\delta,K)\}.\]
	Passing to the $\liminf$ as $p\rightarrow \infty$ and using Lemma \ref{lemma:stimaAutovalorehesimo} we get\[\sum_{\kappa=1}^{h-1} \int_{\frac{1}{K}}^{K} r e^{Z_{m-\kappa}}(\overline\psi_{h}^{m-\kappa})^2 dr \geq -\limsup_{p\rightarrow \infty}\nu_h(p)-3\delta \geq -\beta_{m-h}-3\delta>0,\]
	by the choice of $\delta$, which gives \eqref{etam-1noZero}.

	\

	\
	
	By \eqref{etam-1noZero}, Corollary \ref{remark}-{\it (ii)} implies that
	there exists $A_{h}\neq 0$ such that  
	\begin{eqnarray}\label{AlternativaImpossibile}&&\overline\psi_h^{m-\kappa}=A_h\eta^{m-\kappa}\\
	&&\label{altreAlternativeImpossibili}\overline\psi_h^{i}=0,\qquad  i=0,\ldots, m-1,\ i\neq m-\kappa.\end{eqnarray} 
	
	Furthermore, since by assumption Theorem \ref{thm:autovaloriAutofunzioni} holds true for any  index below $h$, there exists $A_{\kappa}\neq 0$ such that
	\begin{eqnarray}
	&&\label{quellaNonNulla}
	\overline\psi_{\kappa}^{m-\kappa}= A_{\kappa}\eta^{m-\kappa}\\
	\label{quelleNulle}
	&&\overline\psi_{\kappa}^{i}= 0,\qquad i=0,\ldots, m-1, \ i\neq m-\kappa.
	\end{eqnarray}
	We conclude the proof by showing that \eqref{AlternativaImpossibile}
		and \eqref{quellaNonNulla} can not hold at the same time, due to the orthogonality condition \eqref{normalization}.
	
	Observe also that by Lemma \ref{lemma:C1loc},  
	\begin{eqnarray}
	\label{convergenzebuoneKappa}
	 \psi^i_{\kappa,p}\to \overline\psi_{\kappa}^{i}  &&\mbox{ in }C^1_{\loc}(0,+\infty) \mbox{ for }i=1,\dots m-1, \\
	\label{convergenzebuoneKappa0}&&\mbox{  in }C^1_{\loc}[0,+\infty)\mbox{ for }i=0,
	\end{eqnarray} 
	since by assumption  $\bar\nu_{\kappa}:=\lim_{p\rightarrow +\infty}\nu_{\kappa}(p)=\beta^{m-\kappa}$ and by \eqref{ordine-beta} $\beta^{m-\kappa}\leq -25$.

	Then since $\psi_{\kappa,p}\underline{\perp}\psi_{h,p}$, for any $K>1$ and for any $p>1$ we write 
	\begin{eqnarray}
	\label{ortogonalitaDiPartenza}
	&&0=\int_{0}^{1}\frac{\psi_{\kappa,p}\psi_{h,p}}{r} dr \nonumber
	\\
	&&= \int_{G_{p}(K)}\frac{\psi_{\kappa,p}\psi_{h,p}}{r} dr +\int_{0}^{K\varepsilon_{0,p}}
	\frac{\psi_{\kappa,p}\psi_{h,p}}{r} dr  
	+
	\sum_{\substack{i=1\\i\neq m-\kappa}}^{m-1} 
	\int_{\frac{1}{K}\varepsilon_{i,p}}^{K\varepsilon_{i,p}}\frac{\psi_{\kappa,p}\psi_{h,p}}{r} dr + \int_{\frac{1}{K}\varepsilon_{m-\kappa,p}}^{K\varepsilon_{m-\kappa,p}}\frac{\psi_{\kappa,p}\psi_{h,p}}{r} dr
	\nonumber
	\\
	&&= I_{1}(p,K)+ I_{2}(p,K)+I_{3}(p,K)+I_{4}(p,K).\end{eqnarray}

First, as both \eqref{AlternativaImpossibile} and \eqref{quellaNonNulla} hold true, we can take $\d>0$ so that
		\begin{equation}\label{contrad} \delta< \min\left\{ \frac{1}{4}, \frac{|A_{\kappa}A_{h}|}{4} \right\}.
		\end{equation}
		Since  $\int_{0}^{\infty}\frac{(\eta^{m-\kappa})^{2}}{r} dr=1$, there exists $K_{1}(\delta)>1$ such that
		\begin{equation}
		\label{Aintgrande}
		\int_{\frac{1}{K}}^{K}\frac{(\eta^{m-\kappa})^{2}}{r} dr\geq 1-\delta, \qquad \forall K\geq K_{1}(\delta).
		\end{equation}
		Moreover, by  H\"older inequality  and Lemma \ref{Lemma:NormaL2pesataPiccolaSuInsiemini}, we can take $K>K_1(\delta)$ and accordingly $p_1(\delta,K)$ such that
		\begin{equation}\label{I3Zero} |I_{1}(p,K)|:=|\int_{G_{p}(K)}
		\frac{\psi_{\kappa,p}\psi_{h,p}}{r} dr|\leq 
		\left[\int_{G_{p}(K)}\frac{(\psi_{\kappa,p})^{2}}{r}dr\right]^{\frac{1}{2}}
		\left[\int_{G_{p}(K)}\frac{(\psi_{h,p})^{2}}{r}dr\right]^{\frac{1}{2}}\leq \delta
		\end{equation}
		for every $p>p_1(\delta, K)$.

	For the second term we rescale according to the parameter $\varepsilon_{0,p}$ and exploits the $C^{1}_{\loc}[0,\infty)$ convergences of $\psi_{h,p}^{0}$ to $\overline\psi_h^0$ in \eqref{convergenzebuoneh0} and of $\psi_{\kappa,p}^{0}$ to $\overline\psi_{\kappa}^0$ in \eqref{convergenzebuoneKappa0}, we then get
	\begin{eqnarray}\label{I1Zero}|I_{2}(p,K)|&:=&|\int_{0}^{K\varepsilon_{0,p}}
	\frac{\psi_{\kappa,p}\psi_{h,p}}{r} dr |=| \int_{0}^{K}\frac{\psi_{\kappa,p}^{0}\psi_{h,p}^{0}}{r} dr |
	\nonumber
	\\
	&=&|\int_0^K\frac{\overline\psi_{\kappa}^0\overline\psi_h^0}{r}dr|+o_p(1)=o_p(1)\leq \delta, 
	\end{eqnarray}  for any $p\geq p_{2}(\delta,K)$,
	where the last equality follows from the fact that $\overline\psi_{\kappa}^0=0$ by \eqref{quelleNulle}.
	Similarly (scaling with parameter $\varepsilon_{i,p}$ and exploiting the convergences in \eqref{convergenzebuoneh} and  \eqref{convergenzebuoneKappa}) we also get
	\begin{eqnarray}
	\label{I2Zero} 
	|I_{3}(p,K)|&:=&|\sum_{\substack{i=1\\i\neq m-\kappa}}^{m-1} \int_{\frac{1}{K}\varepsilon_{i,p}}^{K\varepsilon_{i,p}}\frac{\psi_{\kappa,p}\psi_{h,p}}{r} dr|=|
	\sum_{\substack{i=1\\i\neq m-\kappa}}^{m-1} \int_{\frac{1}{K}}^{K}\frac{\psi_{\kappa,p}^{i}\psi_{h,p}^{i}}{r} dr|
	\\
	&\leq &\sum_{\substack{i=1\\i\neq m-\kappa}}^{m-1} |\int_{\frac{1}{K}}^{K}\frac{\overline\psi_{\kappa}^i\overline\psi_h^i}{r}dr|+o_p(1)=o_{p}(1)
	\leq \delta ,\end{eqnarray}  for any $p\geq p_{3}(\delta,K)$,
	where the last equality follows from the fact that $\overline\psi_{\kappa}^i=0$, for any $i=1,\ldots, m-1$, $i\neq m-\kappa$ by \eqref{quelleNulle}.  Hence, substituting  \eqref{I3Zero}, \eqref{I1Zero}, \eqref{I2Zero} into \eqref{ortogonalitaDiPartenza}, one gets
	\[|I_{4}(p,K)|\leq 3\delta,\qquad \forall p\geq \max\{p_{1}(\delta),p_{2}(\delta,K),p_{3}(\delta,K)\}.\]
	On the other side, scaling with parameter $\varepsilon_{m-\kappa,p}$, passing to the limit thanks to \eqref{convergenzebuoneh} and  \eqref{convergenzebuoneKappa} with $i=m-\kappa$,
	we also get
	\begin{eqnarray*}
		I_{4}(p,K)&:=&\int_{\frac{1}{K}\varepsilon_{m-\kappa,p}}^{K\varepsilon_{m-\kappa,p}}\frac{\psi_{\kappa,p}\psi_{h,p}}{r} dr
		=
		\int_{\frac{1}{K}}^{K}\frac{\psi_{\kappa,p}^{m-\kappa}\psi_{h,p}^{m-\kappa}}{r} dr
		=
		\int_{\frac{1}{K}}^{K}\frac{\overline\psi_{\kappa}^{m-\kappa}\overline\psi_{h}^{m-\kappa}}{r} dr+o_p(1)
		\\
		&=&
		A_{\kappa}A_{h}\int_{\frac{1}{K}}^{K}\frac{(\eta^{m-\kappa})^{2}}{r} dr+o_{p}(1), 
	\end{eqnarray*}
	as $p\rightarrow +\infty$, where the last equality follows from 
	\eqref{quellaNonNulla} and 
	\eqref{AlternativaImpossibile}.  
	Eventually, passing to the limit for  $p\to\infty$ yields
		\[  |A_{\kappa}A_{h}|\int_{\frac{1}{K}}^{K}\frac{(\eta^{m-\kappa})^{2}}{r} dr\leq 3 \delta , \]
		or equivalenty
		\[  |A_{\kappa}A_{h}| \le \frac{3 \delta}{\int_{\frac{1}{K}}^{K}\frac{(\eta^{m-\kappa})^{2}}{r} dr} \underset{\eqref{Aintgrande}} \le  \frac{3 \delta}{1-\delta}  .\]
		But this last inequality clashes with \eqref{contrad} because
		\[ \frac{3 \delta}{1-\delta} \underset{\delta <  \frac{|A_{\kappa}A_{h}|}4}{<} \frac{3}{4(1-\delta)} |A_{\kappa}A_{h}| \underset{\delta <  \frac{1}4}{<} |A_{\kappa}A_{h}| .\]
		In that way we have reached a contradiction and the proof is completed.
\end{proof}

\

\

\subsection{Last eigenvalue: the proof of Proposition \ref{prop:autovAutofUltimo}}
Here we  prove Proposition \ref{prop:autovAutofUltimo}, thus ending the proof of  Theorem \ref{thm:autovaloriAutofunzioni}. 

\begin{proof}[Proof of Proposition \ref{prop:autovAutofUltimo}]
	Comparing the estimates \eqref{est-rse} and \eqref{est-lim-rse} (for $h=m$) and recalling that $\beta_0=-1$ by \eqref{ordine-beta} yields \[\lim\limits_{p\to+\infty}\nu_m(p) = \beta^0 =-1.\] Proposition \ref{lem:conv-autof-resc} and Corollary \ref{remark}.{\it(ii)} give that
	\begin{align}
	\label{Ai=0} 		 \psi_{m,p}^0 &\to A_m \eta^0  & \\
	\label{0i>0} 		 \psi_{m,p}^i & \to \overline\psi^i_m = 0  & \text{ for } i=1, \dots m-1 , 
	\end{align}
	where the convergence is weak in $\mathcal D_{\mathrm{rad}}$, strong in $L^2_{\loc}(0,\infty)$, and also strong in $C^1_{\loc}(0,\infty)$ thanks to Lemma \ref{lemma:C1loc}.
	It remains to check that the constant $A_m$ in \eqref{Ai=0} is not zero.
	
	Let $\delta>0$, $K=K(\delta)$ and $p\geq p (\delta )$ where $K(\delta)$ and $p(\delta)$ are as in Lemma \ref{lemma-stima-fp} . Following the ideas in   \cite[Proposition 3.5]{AG-N2}, from the equation \eqref{S-L}  we deduce  that 
	\begin{align*}
	-\nu_m(p)&=-\int_0^1 r[(\psi_{m,p}')^2-V_p(\psi_{m,p})^2]dr \le \int_0^1 rV_p(\psi_{m,p})^2dr\\
	& = \int_{G_p(K)}  rV_p(\psi_{m,p})^2 dr+\int_{0}^{K\varepsilon_{0,p}}
	rV_p(\psi_{m,p})^2 dr  +\sum_{i=1}^{m-1} \int_{\frac{1}{K}\varepsilon_{i,p}}^{K\varepsilon_{i,p}} rV_p(\psi_{m,p})^2 dr 
	\\
	&=   I_1(p)+I_2(p)+I_3(p).
	\end{align*}
	Hence   the normalization \eqref{normalization} of the eigenfunction and the estimate obtained in Lemma \ref{lemma-stima-fp} imply 
	\begin{align*}
	I_1(p)=& \int_{G_p(K)}  f_{p} \frac{(\psi_{m,p})^2}{r}dr \le \sup_{G_p(K)}f_p \leq \delta.
	\end{align*}
Furthermore rescaling each integral according to $\e_{i,p}$ gives
	\begin{align*}
	I_3(p)=& \sum_{i=1}^{m-1} \int_{\frac{1}{K}\varepsilon_{i,p}}^{K\varepsilon_{i,p}} rV_p(\psi_{m,p})^2 dr =
	\sum\limits_{i=1}^{m-1} \int_{\frac 1K}^K r  V^i_p (\psi^i_{m,p})^2 dr \le \delta,
	\end{align*}
	for $p\geq p_{3}(\delta,K)$, thanks to \eqref{pot-lim-i} and \eqref{0i>0}.
	\\
	Finally rescaling according to $\e_{0,p}$  and using the uniform convergence in \eqref{pot-lim-0} and the $L^2$ convergence in \eqref{Ai=0} one has
	\begin{align*}
	I_2(p) &= \int_{0}^{K}
	r  V_p^0 (\psi_{m,p}^0)^2 dr \\
	& = (A_m)^2 \int_{0}^{K}
	r e^{Z_0} (\eta^0)^2 dr + o_p(1) \le (A_m)^2 \int_{0}^{K}
	r e^{Z_0} (\eta^0)^2 dr + \delta
	\end{align*} 
	provided that $p\ge p_2(\delta, K)$.
	Notice that Lemma \ref{lemma:C1loc} does not guarantee the convergence in $C^1_{\loc}[0,\infty)$, since $\beta_{0}=-1$.\\	
	Summing up we have showed that
	\[ -\nu_m(p) \le (A_m)^2 \int_{0}^{K}
	r e^{Z_0} (\eta^0)^2 dr + 3\delta ,\]
	provided that $p\ge \max\{p(\d), p_2(\d,K), p_3(\d,K)\}$. Eventually 
	\[ 1= \limsup\limits_{p\to\infty} (-\nu_m(p)) \le (A_m)^2 \int_{0}^{K}
	r e^{Z_0} (\eta^0)^2 dr ,\]
	from which  $A_m\neq 0$ follows.
\end{proof}

\section{The proof of  Theorem \ref{theorem_Main_Intro} in the case $\alpha=0$}

\label{proofMorseLaneEmden}

In this section we compute the exact value of the Morse index of the radial solution $u_{p}$ of the Lane-Emden problem \eqref{LE}, proving that formula \eqref{formulaLaneEmdenIntro} holds if $p$ is sufficiently large.\\

This result follows directly  from  formula \eqref{morse-index-formula} and from  the asymptotic behavior of the singular eigenvalues $\nu_{j}(p)$, $j=1,\ldots, m-1$, as $p\rightarrow +\infty$, which has been stated in Theorem \ref{INTROthm:autovaloriAutofunzioni} (cfr. the more general version Theorem \ref{thm:autovaloriAutofunzioni}).\\

\begin{proof}[Proof of \eqref{formulaLaneEmdenIntro}]
Let $u_{p}$ be the solution to the Lane-Emden problem \eqref{LE} having $m-1$ interior zeros.
From formula \eqref{morse-index-formula} we know that the Morse index $\textsf{m}(u_p)$ is given implicitly  in terms of the negative radial eigenvalues $\nu_{j}(p)$, $j=1,\ldots, m-1$, of the singular problem \eqref{S-L}.  Moreover from Theorem \ref{thm:autovaloriAutofunzioni} we know that 
\[
\sqrt{-\nu_j(p)} \to \frac{\theta_{m-j}}{2} \mbox{ as }p\rightarrow +\infty, \mbox{ for }j=1,\dots m-1.
\]
hence,  recalling \eqref{useful2} we see that
\begin{equation}
\label{finfin} \left\lceil \sqrt{-\nu_j(p)} -1 \right\rceil = \left[\frac{\theta_{m-j}}{2}  \right]=4(m-j)+1 
\end{equation}
	for $p$ large.
The conclusion follows from formula \eqref{morse-index-formula} and \eqref{finfin}, indeed for $p$ large 
\[
 	\textsf{m}(u_{p}) = m  + 2 \sum\limits_{i=1}^{m-1}\left[ \frac{\theta_i }2 \right]. 
 	\]
\end{proof}

\

\section{The proof of Theorem \ref{theorem_Main_Intro} in the case $\alpha>0$}\label{sec:Henon} 
At last we exploit the connection between Lane-Emden and H\'enon problem and show how the already performed analysis allows to compute also the Morse index of radial solutions to the H\'enon problem, concluding the proof of Theorem \ref{theorem_Main_Intro}.
For every $\a>0$, we denote by $u_{\a,p}$ the unique radial solution to 
\begin{equation}\label{H}
\left\{\begin{array}{ll}
-\Delta u = |x|^{\a}|u|^{p-1} u \qquad & \text{ in } B, \\
u= 0 & \text{ on } \partial B,
\end{array} \right.
\end{equation}
with $m$ nodal zones which is positive at the origin. 
In dimension $N=2$ radial solutions to \eqref{H} and \eqref{LE} are linked via the transformation 
\begin{equation}\label{transformation-henon}
u_p(t)= \left(\frac{2}{2+\a}\right)^{\frac{2}{p-1}} u_{\a,p}(r) , \qquad t= r^{\frac{2+\a}{2}},
\end{equation}
where, as in the previous sections, $u_{p}$ denotes the unique radial solution of the Lane-Emden problem with $m$ nodal zones and positive at the origin.
The interested reader can find more details in \cite{GGN16, AG-sing2, IS-arxiv} and the references therein.
The strategy summarized in Section \ref{subsec:Morse} applies also to the H\'enon problem (see \cite{AG-sing1}), indeed the Morse index of $u_{\a,p}$  is equal to the number of the negative eigenvalues $\widehat\Lambda^\a(p)$ of  
\begin{equation}\label{singular-eig-prob-H}
-\Delta \phi-V_{\a,p}(x) \phi = \widehat\Lambda^\a(p) \frac{\phi }{{|x|^2}}  , \qquad \phi \in {\mathcal H}_0(B),
\end{equation}
where now 
\begin{align}\label{V-H}
V_{\a, p}(x) & = p|x|^{\a}| u_{\a,p}(x)|^{p-1}  .
\end{align}
Moreover, similarly as in \eqref{decomposition}, the negative eigenvalues $\widehat\Lambda^\a(p)$ of \eqref{singular-eig-prob-H}  can be decomposed as 
\[\widehat\Lambda^\a(p)=k^{2}+\nu^\a(p),\]
where $\nu^{\a}(p)$ are the eigevalues of the following  singular Sturm-Liouville problem
\begin{equation}\label{S-L-H}
\begin{cases}-(r \, \varphi' )'= r \left( V_{\a,p} + \frac{\nu^{\a}(p)}{r^2} \right) \varphi & \text{ as } 0 < r < 1 , \\
\phi\in {\mathcal H}_{0,\mathrm{rad}}. & \end{cases}
\end{equation} 
Using the transformation $t=r^{\frac{2+\a}{2}}$  one sees that $\varphi_{\a,p}$ is an eigenfunction for \eqref{S-L-H} related to $\nu^{\a}(p)$ if and only if $\psi_p(t)= \varphi_{\a,p}(r)$ is an eigenfunction for \eqref{S-L} related to the eigenvalue
\begin{equation}
\nu(p) = \left(\frac2{2+\a}\right)^2 \nu^{\a}(p), 
\end{equation}
see \cite[Corollary 4.6]{AG-sing2}. Therefore all the results in Sections \ref{subsec:var-car}, \ref{subsec:comp-morse} can be extended also to the Henon problem, in particular 
\begin{equation}\label{est-rseHH}
\nu^{\a}_{1}(p) < \nu^{\a}_{2}(p)<\dots \nu^{\a}_{ m-1}(p)< -\left(\frac{2+\a}{2}\right)^{2} < \nu^{\a}_{m}(p) <0,  
	\end{equation}
and so
\begin{equation}\label{morse-index-formula-H}
\textsf{m}(u_{\alpha, p})=  m + 2 \sum\limits_{j=1}^{ m} \left\lceil \sqrt{-\nu^{\a}_{j}(p)} -1 \right\rceil  =m + 2 \sum\limits_{j=1}^{ m} \left\lceil \frac{2+\a}{2}\sqrt{-\nu_{j}(p)} -1 \right\rceil  ,
\end{equation}
for any $p>1$.


\begin{proof}[Proof of \eqref{formulaHenonIntro1} and \eqref{formulaHenonIntro2}]
The claim follows by inserting the limits computed in Theorem \ref{thm:autovaloriAutofunzioni} inside the Morse index formula \eqref{morse-index-formula-H}.
When $j=m$, using also \eqref{est-rse}, one sees that $ \sqrt{-\nu_j(p)} \to 1$ from below, hence
	\[ \left\lceil \frac{2+\a}2 \sqrt{-\nu_m(p)} -1 \right\rceil =  \left\lceil  \frac{\a}2 \right\rceil \]
for large values of $p$,  since the ceiling function is lower semicontinuous and piecewise constant. Notice that, unlike the Lane-Emden case, also the last eigenvalue $\nu^{\a}_{m}(p) = \left(\frac{2+\a}2\right)^2 \nu_{m}(p)$ gives a contribution to the Morse index.
When $j=1,\dots m-1$ it is only known that
\[ \frac{2+\a}2\sqrt{-\nu_j(p)} \to \frac{2+\a}{4} \theta_{m-j} .\]
If the quantity on the right-hand side is non-integer, it follows that
	\[ \left\lceil \frac{2+\a}2 \sqrt{-\nu_j(p)} -1 \right\rceil =  \left\lceil  \frac{(2+\a)\theta_{m-j}}{4} -1 \right\rceil
	=  \left[\frac{(2+\a)\theta_{m-j}}{4} \right]  \]
	for large values of $p$, and formula \eqref{formulaHenonIntro1} follows. 
Otherwise, only the estimate \eqref{formulaHenonIntro2} can be deduced.
\end{proof}
\begin{remark}[Optimal lower bound for the Morse index]\label{Ederson}
Notice that the Morse index grows quadratically with respect to $m$: indeed in the case $\a=0$  \eqref{formulaLaneEmdenIntro} holds, and in the case $\a>0$ we have that
\begin{align}\nonumber \textsf{m}	(u_{\a,p})  & \ge m + 2 \left\lceil \frac{\a}2\right\rceil+ 2 \sum\limits_{k=1}^{m-1} \left[ \frac{2\!+\!\a }4 \theta_k \right] - 2(m-1) \\ \nonumber
	& \ge m + 2 \left\lceil \frac{\a}2\right\rceil + 2 \sum\limits_{k=1}^{m-1} \left[ \frac{\theta_k}{2} \right]  \left( 1 + \left[ \frac{\a}{2}\right] \right)- 2(m-1) \\ \label{stimaPerConfrontoEderson} 
	& = m +  \left( \textsf{m}	(u_{p}) - \textsf{m}_{\mathrm{rad}}	(u_{p})    \right) \left( 1+ \left[ \frac{\a}{2}\right] \right) + 2 \left( \left\lceil \frac{\a}2\right\rceil -m +1 \right) , \end{align}
where $u_p$ denotes the radial solution to the Lane-Emden problem with the same number of nodal zones.
\\
As already recalled, the lower bound \eqref{MorseLoweBound} is not optimal for the H\'enon problem, even in dimension $N\ge 3$.
 In dimension 2 that lower bound has been recently improved in \cite{dSdS19}, by exploiting the monotonicity of the Morse index with respect to the parameter $\a$, obtaining  that 
\begin{equation}\label{EdersonBound}
	\textsf{m}	(u_{\a,p})  \ge m + \left( \textsf{m}	(u_{0,p}) - \textsf{m}_{\mathrm{rad}}	(u_{0,p})    \right) \left( 1+ \left[ \frac{\a}{2}\right] \right),
\end{equation} 
for any fixed $p>1$ and $\a>0$.
The estimate \eqref{stimaPerConfrontoEderson} shows that neither the lower bound \eqref{EdersonBound} is reached for large values of $p$, at least when $\a>2(m-1)$.
\end{remark}

\section{Further results}

We collect here some further consequences of Theorems \ref{thm:autovaloriAutofunzioni} and \ref{theorem_Main_Intro} that, in our opinion, can bring to a better understanding of both the Lane-Emden and the H\'enon problem in planar domains.

\subsection{Symmetric Morse index}
The decomposition technique used for computing the Morse index allow also to compute suitable symmetric Morse indexes of radial solutions and so, by Morse index comparison, to distinguish among radial solutions and least energy solutions in suitable symmetric spaces, in the spirit of \cite{GI20}. 
The key point is that not only the eigenvalues but also the associated eigenfunctions of the singular eigenvalue problem \eqref{singular-eig-prob} decompose, indeed in radial coordinates they can be written as
\begin{equation}\label{decomp-autof}
	\psi_{j,p}(r) \left( A \cos(k\theta) + B \sin(k\theta) \right),
\end{equation}
where   
\begin{itemize}
	\item $\psi_{j,p}$ is a solution to the singular Sturm-Liouville problem \eqref{S-L} related to $\nu_j(p)$, 
	\item  $\cos(k\theta)$, $\sin(k\theta)$ are the eigenfunctions of the Laplace-Beltrami operator on the circle, related to the eigenvalue $k^2$.
\end{itemize}
Explicit formulas computing the Morse index in symmetric spaces by means of the singular eigenvalues can be found in  \cite[Corollaries 4.3, 4.11]{AG-sing1}. The symmetric Morse index can be computed then, for large values of the parameter $p$, by exploiting  Theorem \ref{thm:autovaloriAutofunzioni}.

\

\subsection{Nondegeneracy for large values of $p$}
It is well known that  the radial solutions are radially non-degenerate, meaning that the linearized problem 
\[ -\Delta w = p|x|^{\a} |u_{\a,p}|^{p-1} w  \]
does not have nontrivial solutions in $H_{0,\mathrm{rad}}(B)$ (see \cite{HRS11} for $\a=0$ and \cite{AG-sing2} for $\a>0$).
Nonradial degeneracy (i.e. existence of solutions in $H_0(B)\setminus H_{0,\mathrm{rad}}(B)$), on the other hand, can be characterized in terms of the singular eigenvalues through the condition 
\begin{equation}\label{degenCond}\nu^{\a}(p) = - k^2 ,\end{equation}
see \cite[Proposition 1.3]{AG-sing1}.
So Theorem \ref{thm:autovaloriAutofunzioni}, together with \eqref{useful2}, yields also that
\begin{corollary}\label{cor:nondegeneracy}
For every positive integer $m$, there exists $p^*>1$ such that radial solutions to the Lane-Emden problem \eqref{LE} with $m$ nodal zones are nondegenerate for $p>p^*$. \\
Moreover for every positive integer $m$ and for every $\a>0$ except at most the sequences $\dfrac{4n}{\theta_i} -2$ (for $i=1, \dots m-1$, $n\in \mathbb N$),  there exists $p^*>1$ such that radial solutions to the H\'enon problem \eqref{H} with $m$ nodal zones are nondegenerate for $p>p^*$. 
\end{corollary}

	\

\subsection{Bifurcation} 
Observe that
Theorem \ref{theorem_Main_Intro} gives  the values of the Morse index \emph{for $p$ large}.
On the other side  one can also compute the Morse index  \emph{when  $p$ is close to $1$},  
by exploiting the (much easier to derive) asymptotic behavior of the radial solutions as $p\rightarrow 1$ from the right (see for instance \cite{BBGV08,Gro09}), and characterizing it in terms of zeros of suitable Bessel functions of the first kind (see \cite{GI20} for the case $(\alpha, N, m)=(0,2,2)$  and \cite{Ama20p1} for the general case). 
\\
As a consequence  one can now detect values of $p\in (1+\infty)$ where the Morse index changes. This is of course a sufficient condition for degeneracy of the solutions at those values of $p$, which could convey to bifurcation from the curve $p\mapsto u_{\a,p}$, for each radial solution  $u_{\a,p}$.
\\
We refer to \cite{GI20} for the case $(\alpha, N, m)=(0,2,2)$ where $3$ branches of bifurcation have been detected, due to a change in the Morse index caused by the first eigenvalue $\nu_{1}(p)$.
For solutions with more nodal regions other eigenvalues may play a role. To give an idea of what may happen, let us consider for instance the case of the \emph{solution $u_{p}$ of the planar Lane-Emden problem ($\alpha=0$) with $m=3$ nodal regions}.  From \cite{Ama20p1} we know that in this case, for $p$ close to 1
\[\nu_{1}(p)\in (-5^{2},-4^{2}),
\qquad
\nu_{2}(p)\in (-3^{2},-2^{2}),\qquad
\nu_{3}(p)\in (-1,0)
\]
while, from Theorem \ref{INTROthm:autovaloriAutofunzioni} one deduces that for $p$ large
\[
\nu_{1}(p)\in (-10^{2},-9^{2}),\qquad
\nu_{2}(p)\in (-6^{2},-5^{2}),\qquad
\nu_{3}(p)\in (-1,0).
\]
As a consequence it follows that 
\[
m(u_{p})=\left\{\begin{array}{lr}15 &\quad\mbox{ for }p\mbox{ close to }1\\
 31 &\quad\mbox{ for }p\mbox{ large}
\end{array}\right.
\]
respectively, and  moreover there exist $p=p_{k}>1$ for $k=3,4,5$ and $p=\hat p_{k}>1$ for $k=5,6,7,8,9$ at which 
%
%
%
the degeneracy condition \eqref{degenCond} is satisfied as follows:
\begin{eqnarray*}
&&\nu_{2}(p_{k})=-k^{2}, \qquad \mbox{ for }k=3,4,5\\
&&\nu_{1}(\hat p_{k})=-k^{2}, \qquad\mbox{ for }k=5,6,7,8,9,
\end{eqnarray*}
thus involving the first two eigenvalues $\nu_{1}(p)$ and $\nu_{2}(p)$. 
Those  $p_{k}, \hat p_{k}$ are the values of $p$ at which one expects that $u_{p}$ bifurcates.  In \cite{FazekasPacellaPlum_ultimo}  some numerical results in this direction have been indeed obtained, see also  \cite[Proposition 4.5]{Ama20bif} where bifurcations at $\hat p_{k}$ (hence from the first eigenvalue) is proved.

\end{document}